\documentclass[11pt]{amsart}
\usepackage{amsmath,amssymb,amsfonts,amsthm,amscd,indentfirst}

\usepackage{hyperref}\usepackage{xcolor}

\def\and{\quad{\rm and}\quad}

\def\<{\langle}
\def\>{\rangle}

\title[Prescribed curvature measure problem]{Prescribed curvature measure problem in hyperbolic space}
\author{Fengrui Yang
        }
\address{Department of Mathematics and Statistics, McGill University, 805 Sherbrooke O, Montreal, Quebec, Canada, H3A 0B9}
\email{fengrui.yang@mail.mcgill.ca}
\email{}
\thanks{Research of the author was supported by CSC fellowship. }

\usepackage{graphicx}
\newtheorem{theorem}{Theorem}
\newtheorem{lemma}{Lemma}

\begin{document}
\begin{abstract}
 The problem of the prescribed curvature measure is one of the important problems in differential geometry and nonlinear partial differential equations. In this paper, we consider prescribed curvature measure problem in hyperbolic space.
We obtain the existence of star-shaped k-convex bodies with prescribed (n-k)-th curvature measures $(k<n)$ by establishing crucial $C^2$ regularity estimates for solutions to the corresponding fully nonlinear PDE in the hyperbolic space.
\end{abstract}
\subjclass{35J60, 35B45}
\date{\today}
\maketitle

\section{Introduction}

This paper concerns the general prescribing curvature measures problem in hyperbolic space $ \mathbb{H}^{n+1}$. Curvature measures
and area measures are two main subjects in convex geometry. They are the local versions of quermassintegrals in the
Brunn-Minkowski theory. They are closely related to the differential geometry and integral geometry of convex surfaces. We first recall the definition of curvature measures and area measures in classical convex geometry in $\mathbb{R}^{n+1}$ (c.f.\cite{15}).
\par
Suppose K is a convex body in $\mathbb{R}^{n+1}$. There are two notions of local parallel sets: given any Borel set $\beta\in \mathcal{B}(\mathbb{R}^{n+1})$, consider
\begin{equation} \label{parallel1}
A_{\rho}(K,\beta)=\{x\in \mathbb{R}^{n+1}|0<d(K,x)\leq\rho, p(K,x)\in\beta\}
\end{equation}
which is the set of all points $x\in\mathbb{R}^{n+1}$ for which the distance $d(K,x)\leq\rho$ and for which the nearest point
$p(K,x)$ belongs to $\beta$. Alternatively, one may prescribe a Borel set $\omega\subset S^{n}$ of unit vectors and then consider
\begin{equation} \label{parallel2}
B_{\rho}(K,\omega)=\{x\in \mathbb{R}^{n+1}|0<d(K,x)\leq\rho, u(K,x)\in\omega\}
\end{equation}
which is the set of all $x\in \mathbb{R}^{n+1}$ for which $d(K,x)\leq\rho$ and for which the unit vector $u(K,x)$ pointing from $p(K,x)$ to $x$ belongs to $\omega$.
\par
A key observation in convex geometry is that the measures of the above local parallel sets are polynomials in the parameter $\rho$.
More precisely, we have the following Steiner formulae in $\mathbb{R}^{n+1}$:
\begin{equation} \label{vol}
Vol(A_{\rho}(K,\beta))=\frac{1}{n+1}\sum_{m=0}^{n}\rho^{n+1-m}C_{n+1}^{m}\cdot\mathcal{C}_{m}(K,\beta)
\end{equation}
\begin{equation} \label{area1}
Vol(B_{\rho}(K,\omega))=\frac{1}{n+1}\sum_{m=0}^{n}\rho^{n+1-m}C_{n+1}^{m}\cdot\mathcal{S}_{m}(K,\omega)
\end{equation}
for $\beta\in \mathcal{B}(\mathbb{R}^{n+1})$, $\omega\in \mathcal{B}(\mathbb{S}^{n})$, and $\rho>0$. Here $C_{n+1}^{m}=\tbinom{n+1}{m}$.

\par
The Steiner formulae (\ref{vol}) and (\ref{area1}) provide excellent controls on the volume and volume growth of parallel sets.
Therefore, the coefficients defined by (\ref{vol}) and (\ref{area1}) yield fundamental geometric information
on the body K. The measure $\mathcal{C}_{0}(K,\cdot)$, $\cdots$, $\mathcal{C}_{n}(K,\cdot)$ are called curvature measures
of the convex body K, and $\mathcal{S}_{0}(K,\cdot)$, $\cdots$, $\mathcal{S}_{n}(K,\cdot)$ are called area measures
of K. It's worth to note that when K is bounded and with $C^{2}$ boundary M. Let $\kappa=(\kappa_{1},\cdots,\kappa_{n})$
be the principal curvatures of M at point x, let $r=(r_{1},\cdots,r_{n})$
be the principal curvature radii, and let $\sigma_{k}$ be the k-th elementary symmetric function. Then we have the following expressions
of m-th curvature measure and area measure:
\begin{equation} \label{curvature1}
\mathcal{C}_{m}(K,\beta)=\frac{1}{C_{n}^{n-m}}\int_{\beta\bigcap M}\sigma_{n-m}(\kappa)d\mu_{g}
\end{equation}
\begin{equation} \label{area2}
\mathcal{S}_{m}(K,\omega)=\frac{1}{C_{n}^{m}}\int_{\omega}\sigma_{m}(r)d\mathbb{S}^{n}
\end{equation}
where $d\mu_{g}$ is the volume element with respect to the induced metric $g$ of M in $\mathbb{R}^{n+1}$, and $d\mathbb{S}^{n}$ is the
volume element of the standard spherical metric.
\par
The Minkowski problem is a problem of prescribing a given n-th area measure. The general Christoffel-Minkowski problem is a problem of prescribing
a given k-th area measure. There is a vast literature devoted to the study of these types of problems, and we refer to
\cite{ex1}, \cite{ex3}, \cite{ex6}, \cite{ex8}, \cite{8}, \cite{ex14}, \cite{ex15}, \cite{ex17}, \cite{ex19}, and \cite{ex12} and the references therein.
We note that the area measures in Euclidean space (\ref{parallel2}), (\ref{area2}) are defined on $\omega\in \mathcal{B}(\mathbb{S}^{n})$ via Gauss map. This implies area measures may not be natural in other space forms due to the invalidity of classical Gauss map. The curvature measures, on the other hand, are defined on the $\beta\in \mathcal{B}(\mathbb{R}^{n+1})$ (\ref{parallel1}), (\ref{curvature1}). So it's possible to study curvature measures in other space forms like $\mathbb{H}^{n+1}$. Our focus in this paper is the corresponding Christoffel-Minkowski problem for the curvature measures, to be specific, the problem of prescribing curvature measures in $\mathbb{H}^{n+1}$.
\par
 The study of curvature measures for more general sets and spaces was carried out by Allendoerfer\cite{2} for  space forms
under strong differentiability assumptions, and by
Federer\cite{5} for sets of positive reach. Sets of positive reach are generalization of convex sets and smooth submanifolds.
Years later, the theory of curvature measures and Steiner formulae for parallel bodies of sets of positive reach in Euclidean space was generalized to space forms by Kohlmann\cite{12}, see also \cite{16}. Besides, Veronelli\cite{16} proved some properties of
curvature measures in hyperbolic space. Therefore, it's natural for us to study curvature measures in the hyperbolic spaces.
We first introduce the definition of curvature measures in $\mathbb{H}^{n+1}$.
\par
Let $\mathbb{K}(\mathbb{H}^{n+1})$ be the set of compact convex sets on $\mathbb{H}^{n+1}$ with non-empty interior. For any $K\in \mathbb{K}(\mathbb{H}^{n+1})$, and
$\rho>0$, define set
\begin{eqnarray*}
K_{\rho}  &=&  \{ x\in \mathbb{H}^{n+1}| d_{\mathbb{H}^{n+1}}(x,K)\leq \rho\}
\end{eqnarray*}

The map $f_{K}: \mathbb{H}^{n+1}\backslash K\rightarrow \partial K$ is defined by
\begin{eqnarray*}
d_{\mathbb{H}^{n+1}}(f_{K}(x),x)=d_{\mathbb{H}^{n+1}}(x,K)
\end{eqnarray*}
and is well-defined because K is convex. For $\beta\in \mathbb{H}^{n+1}$, define also
\begin{eqnarray*}
M_{\rho}(K,\beta)=f_{K}^{-1}(\beta\cap \partial K)\cap (K_{\rho}\backslash K)
\end{eqnarray*}
Following [Kohlmann\cite{12},Allendoerfer\cite{2}], define a Radon measure $\mu_{\rho}$ on the Borel $\sigma$-algebra of hyperbolic space
$\mathcal{B}(\mathbb{H}^{n+1})$ by
\begin{eqnarray*}
\mu_{\rho}(K,\beta)=Vol_{\mathbb{H}^{n+1}}(M_{\rho}(K,\beta))
\end{eqnarray*}
Set $l_{n+1-r}(t)=\int_{0}^{t}\sinh^{n-r}(x)\cosh^{r}(x)dx \qquad r=0,\cdots, n$, then the following Steiner-type formula exists.[\cite{2},\cite{12}].
\begin{equation} \label{curvature2}
\mu_{\rho}(K,\beta)=\sum_{r=0}^{n}l_{n+1-r}(\rho)\Phi_{r}(K,\beta)\qquad \forall\beta\in \mathcal{B}(\mathbb{H}^{n+1})
\end{equation}
When $\eta=\partial K\cap \beta$ is a $C^{3}$ surface, $\Phi_{r}(K,\cdot)$ has following nice expression:
\begin{equation} \label{1}
\Phi_{r}(K,\beta)=\int_{\eta}\sigma_{n-r}^{K}(q)d\mu_{g}(q)
\end{equation}
where $\mu_{g}$ is the surface measure on $\partial K$ induced by $Vol_{\mathbb{H}^{n+1}}$. $\sigma_{n-r}^{K}(q)$
is (n-r)-th elementary symmetric function of the principal curvatures of $\partial K$ at q.
\par
$\Phi_{r}(K,\beta)$ is called r-th curvature measure of convex body K.
\par
The curvature measures defined in (\ref{curvature2}) and (\ref{1}) not only contain information about growth of parallel bodies and surfaces, but
also give insight into topological structures as in the Euclidean case Federer already has observed in \cite{5}. Besides, curvature measures
have strong connections with Euler characteristics. All these imply that the curvature measures defined in (\ref{curvature2}) and (\ref{1})
contain important geometric information. Therefore, it's natural for us to ask prescribed curvature measure problems in space forms like
R.Schneider (see note 8 on page 396 of \cite{15}).
\par
Currently, most work regarding prescribed curvature measure problems were done in Euclidean space. Very little is known in other space forms.
The problem of prescribing 0-th curvature measure is called Alexandrov problem, which is a counter part to Minkowski problem. This problem is equivalent to solve a Monge-Amp$\grave{e}$re type equation on $S^{n}$. In $\mathbb{R}^{n+1}$, the existence and uniqueness were obtained by Alexandrov\cite{1}. The regularity of the Alexandrov problem in elliptic case was proved by Pogorelov\cite{14} for n=2 and by Oliker\cite{13} for higher dimension case.
The general regularity results (degenerate case) of the problem were obtained in Guan-Li\cite{8}. Besides, with certain assumptions of f, Guan-Lin-Ma\cite{9} obtained the existence and regularity of convex solution of prescribed k-th ($k>0$) curvature measure problem.
\par
The general problem of prescribing k-th ($k>0$) curvature
measure is an interesting counterpart of the Christoffel-Minkowski problem, that is, the existence problem of (n-k)-convex solution of prescribed k-th
($k>0$) curvature measure. Guan-Li-Li\cite{7} solved this problem in $\mathbb{R}^{n+1}$ space. But the relevant existence problem of (n-k)-convex solution of prescribed k-th
curvature measure ($\forall k>0$) in hyperbolic space has been open until now. The main contribution of this paper is to resolve this open problem completely.

\medskip

 The corresponding nonlinear PDE of the prescribing curvature measure problem is of the following form,
\begin{equation} \label{2}
\sigma_{k}(\kappa_{1},\cdots,\kappa_{n})=uf
\end{equation}
where u is the support function of hypersurface M.
\par
 This is the same nonlinear PDE as in the case of the prescribing curvature measure problem in $\mathbb R^{n+1}$. Equation (\ref{2}) has the same difficulty as in the case of $\mathbb R^{n+1}$, the involvement of normal vector field $\nu$ on the right hand side. In $\mathbb{R}^{n+1}$, the key is the $C^{2}$ estimate of (\ref{2}) in Guan-Li-Li\cite{7}. The major difference in $\mathbb H^{n+1}$ case is the presence of the negative curvature of the ambient space, which complicates the $C^2$ estimates for equation (\ref{2}). To be more specific, term $Kh_{11}$ exists when switch $h_{11ii}$ into $h_{ii11}$, where K is the sectional curvature of space form.
In $\mathbb{R}^{n+1}$, K=0, this term doesn't exist. Yet in hyperbolic space, this term becomes $-h_{11}$, a problematic term that need to be handled.
\par
 If the right hand side of equation (\ref{2}) is independent of $\nu$, that is the equation of prescribed curvature.  Jin-Li\cite{11} studied
  this type of equation and solved prescribed curvature problem in $\mathbb{H}^{n+1}$. They gave a nice method in handling $C^{2}$ estimate of the prescribed curvature equation from which we benefit a lot.
 Recently, a very important result was made by Guan-Ren-Wang\cite{10}. They gave the $C^{2}$ estimate for convex solution of general curvature measure equation
\begin{eqnarray} \label{222}
\sigma_{k}(\kappa_{1},\cdots,\kappa_{n})=f(x,\nu(x))
\end{eqnarray}
and also a proof for 2-convex solution in the case $k=2$. Besides, Ren-Wang\cite{RW} obtained $C^{2}$ estimate for (n-1)-convex solution of (\ref{222}) in the case $k=n-1$.
In the case of $k=2$, Spruck-Xiao \cite{spruckX} obtained $C^2$ estimates for solutions of equation (\ref{222}) in general space forms. Thus, $C^2$ estimate for equation (\ref{2}) is verified for $k=2$ in $\mathbb H^{n+1}$. Their paper\cite{spruckX} not only gave a simpler proof for the scalar curvature case of Guan-Ren-Wang\cite{10}, but also have a large impact on our $C^{2}$ estimate.
\par
In this paper, we will present a $C^{2}$ estimate for the general k-convex solutions of prescribed curvature measure equation (\ref{2}) ($\forall1\leq k\leq n$).

\medskip

We note the curvature measures defined in (\ref{1}) only require $\partial K$ to be $C^{3}$. This implies we do not necessarily ask K to be convex if $r>0$ in (\ref{1}). And it's indeed possible to study curvature measure for more general sets. In the work of Alexandrov\cite{1}, the curvature measures are prescribed on $S^{n}$ via a radial map. Under the radial parametrization of $\Omega$, the star-shaped domains are natural class for us to study. So in the rest of this paper, we will prove existence theorem of the prescribing general k-th curvature measure problem with $k>0$ on bounded $C^{2}$ star-shaped domains in hyperbolic space.

\medskip

Let M be a bounded star-shaped domain. Therefore, it can be parametrized by a graph $\Sigma=\{ (\rho(\theta),\theta)|\theta\in S^{n}\}$.
Denote
\begin{eqnarray} \label{3}
 R_{M} & & S^{n}\rightarrow M \nonumber  \\
       & & \theta\rightarrow (\rho(\theta),\theta)
\end{eqnarray}
The (n-k)-th curvature measure of $\mathbb{H}^{n+1}$ space on each Borel set $\beta$ in $S^{n}$ can be defined as
\begin{equation} \label{4}
C_{n-k}(M,\beta)=\int_{R_{M}(\beta)}\sigma_{k}(\kappa)d\mu_{g}
\end{equation}
\par
Then the problem of prescribing (n-k)-th curvature measure is:

\medskip

{\it Question:} Given a positive function $f\in C^{2}(S^{n})$, find a closed hypersurface M which can be parametrized like (\ref{3}), such that
\begin{equation} \label{5}
C_{n-k}(M,\beta)=\int_{\beta}fd\mu_{S^{n}}
\end{equation}
 for every Borel set $\beta$ in $S^{n}$.
\par
As later we shall see, due to parametrization (\ref{3}), the prescribed curvature measure problem for star-shaped domain can be reduced
to the following curvature type nonlinear partial differential equation of $\rho$ on $S^{n}$:
\begin{equation} \label{6}
\sigma_{k}(\kappa_{1},\cdots,\kappa_{n})=\frac{f}{\phi(\rho)^{n-1}\sqrt{\phi(\rho)^{2}+|\nabla \rho|^{2}}}
\end{equation}
Here $\kappa=(\kappa_{1},\cdots,\kappa_{n})$ is principal curvature vector of M, $\phi(\rho)=\sinh(\rho)$. We say M is k-convex if
$\kappa(x)=(\kappa_{1}(x),\cdots,\kappa_{n}(x))\in\Gamma_{k}$ at every point $x\in M$. It's also worth
to note that equation (\ref{6}) is a special type of fully nonlinear partial differential equation studied in the pioneer work by
Caffarelli-Nirenberg and Spruck\cite{3}.
\par
We now state our main theorem:

\medskip

\begin{theorem} \label{T0}
 Let $n\geq 2$, $1\leq k\leq n-1$. Suppose $f\in C^{2}(S^{n})$ and $f>0$. Then
there exists a unique k-convex star-shaped hypersurface $M\in C^{3,\alpha}$, such that it satisfies (\ref{6}).
\par
Moreover, there is a constant C only depending on k,n,$||f||_{C^{2}}$, $inf(f)$, and $\alpha$, such that
\begin{eqnarray*}
||\rho||_{C^{3,\alpha}}\leq C
\end{eqnarray*}
\end{theorem}

\medskip

The crucial part of proving theorem \ref{T0} is to establish $C^{2}$ estimate for (\ref{2}) or (\ref{6}). Here we use the trick from Brendle-Choi-Daskalopoulos\cite{brendle} where they use smooth function to approximate the smallest eigenvalue $\lambda_{n}$. We use the same method to approximate the largest eigenvalue $\lambda_{1}$. The advantage of $\lambda_{1}$ is that it has extra good third-order derivative terms after differentiate
it twice. These extra good third-order terms are extremely important in our $C^{2}$ estimate. After that, we separate several cases carefully and finish $C^{2}$ estimate step by step.

\medskip

Theorem \ref{T0} gives a complete proof for the existence and uniqueness of star-shaped k-convex body with prescribed (n-k)-th curvature measure
$(k<n)$. When $k=n$, this is the Alexandrov problem in hyperbolic space, i.e. the prescribed 0-th curvature measure problem. At this case, the only difficulty comparing to Theorem \ref{T0} is the lack of lower positive $C^{0}$ estimate. In fact, as we shall see in Theorem \ref{T2}, if
$\max_{S^{n}}(f)<1$, there is no solution which satisfies (\ref{6}) ($k=n$, $\kappa\in \Gamma_{n}$). Therefore, in the prescribed 0-th curvature measure
problem, extra condition on $f$ is necessary! In section 5, we will give a proof of the existence and uniqueness of convex body with prescribed 0-th curvature measure under extra condition $\inf_{S^{n}}(f)>1$. We note that the relevant existence problem $(k=n)$, when $f$ is endowed with other
geometric conditions, is an interesting remaining question.

\medskip

The rest of this paper is organized as follows. In section 2, we state some useful properties regarding elementary symmetric function. In section 3, we derive equation (\ref{6}) and establish $C^{0}$ and $C^{1}$ estimates. In section 4, we establish the crucial $C^{2}$ estimate.
In the section 5, we finish the proof of the existence and uniqueness of prescribed curvature measure problem in hyperbolic space.

\medskip

We would like to thank Prof. Junfang Li for all his inspiring discussions and helpful comments. We are extremely grateful to Prof. Pengfei Guan for his supervision and all his important advice.

\section{Preliminaries}

In this section, we state some Lemmas regarding elementary symmetric functions which we are going to use in the following sections.
\par
For $1\leq k\leq n$, $\lambda=(\lambda_{1},\cdots,\lambda_{n})\in \mathbb{R}^{n}$,
\begin{eqnarray*}
\sigma_{k}(\lambda)=\Sigma_{1\leq i_{1}<i_{2}<\cdots <i_{k}\leq n}\underline{}\lambda_{i_{1}}\lambda_{i_{2}}\cdots\lambda_{i_{k}}
\end{eqnarray*}
\par
For $n\times n$ symmetric matrix, $W=\{W_{ij}\}$. Let $\lambda(W)=(\lambda_{1}(W),\cdots,\lambda_{n}(W))$ be eigenvalues of W. Then define
\begin{eqnarray*}
\sigma_{k}(W)=\sigma_{k}(W_{ij})=\sigma_{k}(\lambda(W))=\Sigma_{1\leq i_{1}<i_{2}<\cdots <i_{k}\leq n}\lambda_{i_{1}}(W)\lambda_{i_{2}}(W)\cdots\lambda_{i_{k}}(W)
\end{eqnarray*}
\par
Let $\sigma_{k}(W|i)$ be the symmetric function with W deleting the i-row and i-column and
$\sigma_{k}(W|ij)$ be the symmetric function with W deleting the i,j-row and i,j-column.
\par
\begin{lemma} \label{lemma1}
 Suppose $W=\{W_{ij}\}$ is diagonal. $1\leq k\leq n$. Then
\begin{eqnarray*}
\frac{\partial\sigma_{k}(W)}{\partial W_{ij}}&=&\left\{
\begin{aligned}
\sigma_{k-1}(W|i)&& if&& i=j\\
0 && if && i\neq j
\end{aligned}
\right.
\\
\frac{\partial^{2}\sigma_{k}(W)}{\partial W_{ij}\partial W_{sl}}&=&\left\{
\begin{aligned}
\sigma_{k-2}(W|is)&& if&& i=j,s=l,i\neq s\\
-\sigma_{k-2}(W|is)&& if&& i=l,j=s,i\neq j\\
0 && && otherwise
\end{aligned}
\right.
\end{eqnarray*}
\par
\end{lemma}
$\lambda=(\lambda_{1}(W),\cdots,\lambda_{n}(W))$. When W is diagonal, suppose $\lambda_{i}=W_{ii}$, then
\begin{eqnarray*}
\sigma_{k-1}(W|i)&=&\sigma_{k-1}(\lambda|i)=\frac{\partial\sigma_{k}(\lambda)}{\partial\lambda_{i}}\\
\sigma_{k-2}(W|ij)&=&\sigma_{k-2}(\lambda|ij)=\frac{\partial^{2}\sigma_{k}(\lambda)}{\partial\lambda_{i}\partial\lambda_{j}}
\end{eqnarray*}
\par
$\mathbf{Definition}$: For $1\leq k\leq n$
\begin{eqnarray*}
\Gamma_{k}=\{\lambda\in \mathbb{R}^{n}|\sigma_{j}(\lambda)>0,\forall j=1,\cdots,k\}
\end{eqnarray*}
A $n\times n$ symmetric matrix W belongs to $\Gamma_{k}$ if $\lambda(W)\in\Gamma_{k}$.

\medskip

\begin{lemma} \label{lemma2}
 Let $\lambda\in\Gamma_{k}$ and suppose $\lambda_{1}\geq\lambda_{2}\geq\cdots\geq\lambda_{n}$, then
 \begin{equation} \label{2.1}
\lambda_{1}\sigma_{k-1}(\lambda|1)\geq c(n,k)\sigma_{k}(\lambda)
\end{equation}
For $k\geq l\geq 1$
\begin{equation} \label{2.2}
(\frac{\sigma_{k}(\lambda)}{C_{n}^{k}})^{\frac{1}{k}}\leq(\frac{\sigma_{l}(\lambda)}{C_{n}^{l}})^{\frac{1}{l}}
\end{equation}
here $c(n,k)$ means a constant only depending on n,k.
\end{lemma}
\par
The following Lemma shows the uniformly elliptic of $\sigma_{k}$ operator on the condition that we have $C^{2}$ bound.
\par
\begin{lemma} \label{lemma3}
 Let $F=\sigma_{k}$, then the matrix $\{\frac{\partial F}{\partial W_{ij}}\}$ is positive definite for $W\in\Gamma_{k}$.
\par
Furthermore, if $||W||=\sqrt{\sum_{i,j}W_{ij}^{2}}\leq R$, then we have
\begin{eqnarray*}
\frac{\sigma_{k}(W)}{R(1+c_{n,k}\cdot\sigma_{k-1}^{\frac{1}{k-1}}(I))}I\leq \{\frac{\partial F}{\partial W_{ij}}\}
\leq R^{k-1}\sigma_{k-1}(I)I
\end{eqnarray*}
\end{lemma}
\par
The proof is from Guan \cite{6}.

\medskip

\begin{lemma}\label{lemma4}
 If $W\in\Gamma_{k}$, then $\{\frac{\partial\sigma_{k}^{\frac{1}{k}}}{\partial W_{ij}}\}$ is positive definite and
$\sigma_{k}^{\frac{1}{k}}(W)$ is a concave function in $\Gamma_{k}$.
\end{lemma}

\medskip

\begin{lemma} \label{lemma5}
 $\lambda=(\lambda_{1},\cdots,\lambda_{n})$ for $k>l\geq 0$.
We have $(\frac{\sigma_{k}(\lambda)}{\sigma_{l}(\lambda)})^{\frac{1}{k-l}}$ is a concave function in $\Gamma_{k}$. i.e.
\begin{eqnarray*}
\sum_{i,j}\frac{\partial^{2}(\frac{\sigma_{k}(\lambda)}{\sigma_{l}(\lambda)})^{\frac{1}{k-l}}}{\partial\lambda_{i}\partial\lambda_{j}}
\xi_{i}\xi_{j}\leq 0\qquad \forall \xi=(\xi_{1},\cdots,\xi_{n})\in \mathbb{R}^{n}
\end{eqnarray*}
This is equivalent as
\begin{eqnarray*}
-\sum_{i\neq j}\sigma_{k-2}(\lambda|ij)\xi_{i}\xi_{j}&\geq &(1+\frac{1}{k-l})\frac{\sigma_{k}}{\sigma_{l}^{2}}
(\sum_{i}\frac{\partial\sigma_{l}}{\partial\lambda_{i}}\xi_{i})^{2}-\frac{2}{k-l}\frac{1}{\sigma_{l}}\cdot(\sum_{j}\frac{\partial\sigma_{k}}{\partial\lambda_{j}}\xi_{j})\cdot
(\sum_{i}\frac{\partial\sigma_{l}}{\partial\lambda_{i}}\xi_{i})\\
& &-(1-\frac{1}{k-l})\frac{1}{\sigma_{k}}
(\sum_{j}\frac{\partial\sigma_{k}}{\partial\lambda_{j}}\xi_{j})^{2}
-\sigma_{k}\frac{\sum_{i\neq j}\sigma_{l-2}(\lambda|ij)\xi_{i}\xi_{j}}{\sigma_{l}}
\end{eqnarray*}
\end{lemma}

\medskip

\begin{lemma}\label{lemma6}
 For $\forall i\neq j$,
 \begin{eqnarray*}
\sigma_{l}^{ii}\sigma_{l}^{jj}-\sigma_{l}\sigma_{l-2}(\lambda|ij)=\sigma_{l-1}(\lambda|ij)^{2}
-\sigma_{l}(\lambda|ij)\sigma_{l-2}(\lambda|ij)
\end{eqnarray*}
\begin{proof}
 \begin{eqnarray*}
 \sigma_{l}&=&\sigma_{l-1}(\lambda|i)\lambda_{i}+\sigma_{l}(\lambda|i)\\
&=&(\sigma_{l-1}(\lambda|ij)+\lambda_{j}\sigma_{l-2}(\lambda|ij))\lambda_{i}+\sigma_{l}(\lambda|ij)+\sigma_{l-1}(\lambda|ij)\lambda_{j}\\
&=&\sigma_{l-1}(\lambda|ij)(\lambda_{i}+\lambda_{j})+\lambda_{i}\lambda_{j}\sigma_{l-2}(\lambda|ij)+\sigma_{l}(\lambda|ij)
\end{eqnarray*}
Therefore
\begin{eqnarray*}
 \sigma_{l}^{ii}\sigma_{l}^{jj}-\sigma_{l}\sigma_{l-2}(\lambda|ij)&=&\sigma_{l-1}(\lambda|i)\sigma_{l-1}(\lambda|j)
-\sigma_{l}\sigma_{l-2}(\lambda|ij)\\
& =&(\sigma_{l-1}(\lambda|ij)+\lambda_{j}\sigma_{l-2}(\lambda|ij))(\sigma_{l-1}(\lambda|ij)+\lambda_{i}\sigma_{l-2}(\lambda|ij))\\
& & -(\sigma_{l-1}(\lambda|ij)(\lambda_{i}+\lambda_{j})+\lambda_{i}\lambda_{j}\sigma_{l-2}(\lambda|ij)+\sigma_{l}(\lambda|ij))
\sigma_{l-2}(\lambda|ij)\\
&=&\sigma_{l-1}(\lambda|ij)^{2}-\sigma_{l}(\lambda|ij)\sigma_{l-2}(\lambda|ij)
\end{eqnarray*}
\end{proof}
\end{lemma}

\section{$C^{0}$ and $C^{1}$ estimate}

Let $(M,g)$ be a hypersurface in $\mathbb{H}^{n+1}$ (space form with constant sectional curvature -1) with induced metric g.
\par
M is a bounded star-shaped domain. We can parametrize M over $S^{n}$ by positive function $\rho$. Due to this parametrization,
the prescribed curvature measure problem for this class of domains can be reduced to a curvature type nonlinear partial differential equation
of $\rho$ on $S^{n}$.
\par
We now give the following geometric condition on M.
\par
$\mathbf{Definition}$ $\quad$ We say a smooth hypersurface $M\in \mathbb{H}^{n+1}$ is k-convex for some $0\leq k\leq n$ if its principal curvature vector
$\kappa(x)\in \Gamma_{k}, \quad \forall x\in M$
$~\\$
where $\Gamma_{k}$ is the Garding cone defined by
$\Gamma_{k}=\{\lambda\in \mathbb{R}^{n}|\sigma_{j}(\lambda)>0,\quad \forall j=1,\cdots k\}$
\par
Since M is star-shaped, it can be parametrized by a graph $\Sigma=\{ (\rho(\theta),\theta)|\theta\in S^{n}\}$. Denote
\begin{eqnarray} \label{a.1}
 R_{M} & & S^{n}\rightarrow M   \\
       & & \theta\rightarrow (\rho(\theta),\theta) \notag
\end{eqnarray}

From Veronelli \cite{16}, Kohlmann \cite{12}, the (n-k)-th curvature measure of $\mathbb{H}^{n+1}$ space on each Borel set $\beta$ in $S^{n}$ can be defined as
\begin{equation} \label{a.2}
C_{n-k}(M,\beta)=\int_{R_{M}(\beta)}\sigma_{k}(\kappa)d\mu_{g}
\end{equation}
\par
Then the problem of prescribing (n-k)-th curvature measure is:

\medskip

Given a positive function $f\in C^{2}(S^{n})$, find a closed hypersurface M which can be parametrized like (\ref{a.1}), such that
\begin{equation} \label{a.3}
C_{n-k}(M,\beta)=\int_{\beta}fd\mu_{S^{n}}
\end{equation}
 for every Borel set $\beta$ in $S^{n}$.
\par
Since M's induce metric is g, therefore, density function is $\sqrt{det(g)}$. We have
\begin{equation} \label{a.4}
C_{n-k}(M,\beta)=\int_{R_{M}(\beta)}\sigma_{k}(\kappa)d\mu_{g}=\int_{\beta}\sigma_{k}\cdot\sqrt{det(g)}d\mu_{S^{n}}
\end{equation}
\par
Let $\{e_{1},\cdots, e_{n}\}$ be a local orthonormal frame on $S^{n}$. Denote $e_{ij}$ the standard spherical metric with respect
to the frame. All the covariant derivatives with respect to the standard spherical metric $e_{ij}$ will also be denoted as $\nabla$ when there is no confusion in the context.
\par

Under the Gaussian geodesic normal coordinates, the metric of $\mathbb{H}^{n+1}$ can be expressed as
\begin{eqnarray*}
ds^{2}=d\rho^{2}+\phi(\rho)^{2}dz^{2}
\end{eqnarray*}
Here $\phi(\rho)=\sinh(\rho),\quad \rho\in [0,\infty)$, and $dz^{2}$ is the induced standard metric on $S^{n}$ in Euclidean space.
We define
\begin{eqnarray*}
\Phi(\rho)=\int_{0}^{\rho}\phi(s)ds
\end{eqnarray*}
Consider the vector field $V=\phi(\rho)\frac{\partial}{\partial\rho}$, and let $\nu$ be the outward normal vector field of M.
Then the generalized support function of M is defined as $u=\langle V,\nu\rangle$.

\par
Because M is a star-shaped hypersurface, the support function, induced metric, inverse metric matrix, second fundamental form can be expressed as follows:
\begin{eqnarray*}
& &u=\frac{\phi^{2}}{\sqrt{\phi^{2}+|\nabla \rho|^{2}}}\\
& & g_{ij}=\phi^{2}\delta_{ij}+\rho_{i}\rho_{j},\quad
g^{ij}=\frac{1}{\phi^{2}}(e^{ij}-\frac{\rho_{i}\rho_{j}}{\phi^{2}+|\nabla \rho|^{2}})\\
& &h_{ij}=\frac{1}{\sqrt{\phi^{2}+|\nabla \rho|^{2}}}(-\phi\rho_{ij}+2\phi^{'}\rho_{i}\rho_{j}+\phi^{2}\phi^{'}\delta_{ij})\\
& &h_{j}^{i}=\frac{1}{\phi^{2}\sqrt{\phi^{2}+|\nabla \rho|^{2}}}(e^{ik}-\frac{\rho_{i}\rho_{k}}{\phi^{2}+|\nabla \rho|^{2}})(-\phi\rho_{jk}+2\phi^{'}\rho_{k}\rho_{j}+\phi^{2}\phi^{'}\delta_{kj})\\
& &\tilde{h_{j}^{i}}=\frac{1}{\phi^{2}\sqrt{\phi^{2}+|\nabla \rho|^{2}}}(\delta^{ik}-\frac{\rho_{i}\rho_{k}}{\tilde{\omega}(\tilde{\omega}+\phi)})(-\phi\rho_{kl}+2\phi^{'}\rho_{k}\rho_{l}+\phi^{2}\phi^{'}\delta_{kl})(\delta^{lj}-
\frac{\rho_{l}\rho_{j}}{\tilde{\omega}(\tilde{\omega}+\phi)})
\end{eqnarray*}

\par
Here $\phi$ means $\phi(\rho)$, $\rho_{i}$ is the derivative with respect to spherical metric $e_{ij}$ and $\tilde{\omega}=\sqrt{\phi^{2}+|\nabla \rho|^{2}}$.
\par

Then the principal curvatures $(\kappa_{1},\cdots,\kappa_{n})$ of M are the eigenvalues of symmetric matrix $H=(\tilde{h_{j}^{i}})$ and
\begin{equation} \label{a.5}
\sqrt{det(g_{ij})}=\phi^{n-1}\cdot \sqrt{\phi^{2}+|\nabla \rho|^{2}}
\end{equation}
\par
Therefore, from (\ref{a.3}) and (\ref{a.4}),  the prescribed curvature measure problem can be reduced to the following curvature measure equation on $S^{n}$
\begin{equation} \label{a.6}
\sigma_{k}(\kappa_{1},\cdots,\kappa_{n})=\sigma_{k}(\tilde{h_{j}^{i}})=\frac{f}{\phi^{n-1}\sqrt{\phi^{2}+|\nabla \rho|^{2}}}
\end{equation}
\par
Here $f>0$ is the given function defined on $S^{n}$. We say a solution of (\ref{a.6}) is admissible if $\kappa(X)=(\kappa_{1},\cdots,\kappa_{n})
\in \Gamma_{k},\qquad \forall X\in M$.
\par
Actually, any positive $C^{2}$ function $\rho$ on $S^{n}$ satisfying equation (\ref{a.6}) is an admissible solution. This is because at the point where $\rho$ obtains its maximum, we have $\nabla\rho=0$, then
\begin{eqnarray*}
\tilde{h_{j}^{i}}=\frac{1}{\phi^{3}}(-\phi\rho_{ij}+\phi^{2}\phi^{'}\delta_{ij})
\end{eqnarray*}
\par
Since matrix $\{\rho_{ij}\}$ is semi-negative definite at this point, then all principal curvatures are positive, which means the solution is admissible
at this point.
\par
Because $\Gamma_{k}, S^{n}$ are connected, $\kappa(X)$ is continuous $(X\in M)$ and the fact that $\sigma_{k}(\lambda)=0$ on $\partial\Gamma_{k}$, we obtain this solution is admissible at any point of M.
\par
Next, we will prove the main theorem.
\par
\begin{theorem} \label{T1}
 Let $n\geq 2$, $1\leq k\leq n-1$. Suppose $f\in C^{2}(S^{n})$ and $f>0$. Then
there exists a unique k-convex star-shaped hypersurface $M\in C^{3,\alpha}$, such that it satisfies (\ref{a.6}).
\par
Moreover, there is a constant C only depending on k,n,$||f||_{C^{2}}$, $inf(f)$, and $\alpha$, such that
\begin{eqnarray*}
||\rho||_{C^{3,\alpha}}\leq C
\end{eqnarray*}
\end{theorem}

\medskip

The $C^{0}$ and $C^{1}$ estimates are the same methods with Guan-Li-Li\cite{7}. Even though, we still give the proof here for completeness.
\par
We first prove $C^{0}$ estimate.
\par
\begin{theorem} \label{T2}
 let $n\geq 2$, $1\leq k\leq n-1$. Suppose $\rho$ is a solution of (\ref{a.6}), then
 \begin{eqnarray*}
c_{0}\leq \min(\rho)\leq \max(\rho)\leq c_{1}
\end{eqnarray*}
$c_{0}$, $c_{1}\sim \inf(f)$, $|f|_{C^{0}}$, $n$, $k$.
\end{theorem}
\begin{proof} At the point where $\rho$ obtains its maximum, we have $\nabla\rho=0$, then
\begin{eqnarray*}
\tilde{h_{j}^{i}}=-\frac{\rho_{ij}}{\phi^{2}}+\frac{\phi^{'}}{\phi}\delta_{ij}
\end{eqnarray*}
\par
Let $F(t)=\sigma_{k}(-t\cdot\frac{\rho_{ij}}{\phi^{2}}+\frac{\phi^{'}}{\phi}\delta_{ij})$, then
\begin{eqnarray*}
\sigma_{k}(-\frac{\rho_{ij}}{\phi^{2}}+\frac{\phi^{'}}{\phi}\delta_{ij})-\sigma_{k}(\frac{\phi^{'}}{\phi}\delta_{ij})
&=&F(1)-F(0)=\int_{0}^{1}F^{'}(s)ds\\
&=&\sum_{i,j}(\int_{0}^{1}\sigma_{k}^{ij}(-s\cdot\frac{\rho_{st}}{\phi^{2}}+\frac{\phi^{'}}{\phi}\delta_{st})ds)\cdot\frac{-\rho_{ij}}{\phi^{2}}
\end{eqnarray*}
\par
Since $\rho$ is an admissible solution, we have $\{-\frac{\rho_{ij}}{\phi^{2}}+\frac{\phi^{'}}{\phi}\delta_{ij}\}\in\Gamma_{k}$.
Because $\{\frac{\phi^{'}}{\phi}\delta_{ij}\}\in\Gamma_{k}$ and the fact that $\Gamma_{k}$ is convex, we have $
\{-s\cdot\frac{\rho_{ij}}{\phi^{2}}+\frac{\phi^{'}}{\phi}\delta_{ij}\}\in\Gamma_{k},\quad \forall s\in [0,1]$.
\par
This implies $\{\sigma_{k}^{ij}(-s\cdot\frac{\rho_{st}}{\phi^{2}}+\frac{\phi^{'}}{\phi}\delta_{st})\}$ is positive definite. $ \forall s\in [0,1]$.
Hence, at the maximum point,
\begin{eqnarray*}
\sigma_{k}(-\frac{\rho_{ij}}{\phi^{2}}+\frac{\phi^{'}}{\phi}\delta_{ij})\geq\sigma_{k}(\frac{\phi^{'}}{\phi}\delta_{ij})&=& C_{n}^{k}
\cdot\frac{(\phi^{'})^{k}}{\phi^{k}}\\
f=\phi^{n}\cdot\sigma_{k}(-\frac{\rho_{ij}}{\phi^{2}}+\frac{\phi^{'}}{\phi}\delta_{ij})&\geq & C_{n}^{k}
\cdot(\phi^{'})^{k}\cdot\phi^{n-k}\geq C_{n}^{k}
\cdot \phi^{n}
\end{eqnarray*}
\par
We have the upper estimate of $\rho$.
\par
Similarly, at the minimal point of $\rho$, we have
\begin{eqnarray*}
f\leq C_{n}^{k}\cdot(\phi^{'})^{k}\cdot\phi^{n-k}
\end{eqnarray*}
Since $k<n$, we obtain the positive lower bound for $\rho$.
\end{proof}

Before the proof of $C^{1}$ estimate, we first introduce a new variable $\gamma$, satisfying
\begin{eqnarray*}
\frac{d\gamma}{d\rho}=\frac{1}{\phi}
\end{eqnarray*}
Define $\omega=\sqrt{1+|\nabla \gamma|^{2}}$, then we have
\begin{eqnarray*}
& &g_{ij}=\phi^{2}(\delta_{ij}+\gamma_{i}\gamma_{j}), g^{ij}=\frac{1}{\phi^{2}}(\delta^{ij}-\frac{\gamma_{i}\gamma_{j}}{\omega^{2}})\\
& &h_{ij}=\frac{\phi}{\omega}(-\gamma_{ij}+\phi^{'}\gamma_{i}\gamma_{j}+\phi^{'}\delta_{ij})\\
& &\tilde{h_{j}^{i}}=\frac{1}{\phi\omega}(\delta^{ik}-\frac{\gamma_{i}\gamma_{k}}{\omega(\omega+1)})
(-\gamma_{kl}+\phi^{'}\gamma_{k}\gamma_{l}+\phi^{'}\delta_{kl})(\delta^{lj}-\frac{\gamma_{l}\gamma_{j}}{\omega(\omega+1)})
\end{eqnarray*}
Then equation (\ref{a.6}) becomes
\begin{equation} \label{a.7}
\frac{\phi^{n-k}}{\omega^{k-1}}\sigma_{k}(b_{j}^{i})=f
\end{equation}
Here $\phi$ means $\phi(\rho)$ and
\begin{eqnarray*}
b_{j}^{i}=(\delta^{ik}-\frac{\gamma_{i}\gamma_{k}}{\omega(\omega+1)})
(-\gamma_{kl}+\phi^{'}\gamma_{k}\gamma_{l}+\phi^{'}\delta_{kl})(\delta^{lj}-\frac{\gamma_{l}\gamma_{j}}{\omega(\omega+1)})
\end{eqnarray*}
\par
Apparently, $(\lambda_{1}(b_{j}^{i}),\cdots,\lambda_{n}(b_{j}^{i}))\in\Gamma_{k}$, here $\lambda_{i}$ means the eigenvalue of matrix $\{b_{j}^{i}\}$.
\par
We now prove $C^{1}$ estimate.
\par
\begin{theorem} \label{T3}
 let $n\geq 2$, $1\leq k\leq n-1$. Suppose $\rho$ is a solution of (\ref{a.6}). Then
 \begin{eqnarray*}
\max |\nabla\rho|\leq C_{2}
\end{eqnarray*}
$C_{2}\sim \inf(f)$, $||f||_{C^{1}}$, $n$, $k$.
\end{theorem}

\begin{proof}  We only need to prove $|\nabla\gamma|$ is bounded.
\par
Define $G_{\alpha s}=\delta^{\alpha s}-\frac{\gamma_{\alpha}\gamma_{s}}{\omega(\omega+1)}$, then matrix $G=\{G_{\alpha s}\}$
and matrix $A=(\frac{\partial\sigma_{k}}{\partial b_{\beta}^{\alpha}})$ are symmetric positive definite. Also define
\begin{eqnarray*}
\widetilde{F^{st}}=\sum_{\alpha,\beta}\frac{\partial\sigma_{k}}{\partial b_{\beta}^{\alpha}}
(\delta^{\alpha s}-\frac{\gamma_{\alpha}\gamma_{s}}{\omega(\omega+1)})(\delta^{t \beta}-\frac{\gamma_{t}\gamma_{\beta}}{\omega(\omega+1)})
\end{eqnarray*}
Since matrix $F=\{\widetilde{F^{st}}\}=G\cdot A\cdot G$, we have $\{\widetilde{F^{st}}\}$ is positive definite.
\par
Consider test function $\frac{1}{2}|\nabla\gamma|^{2}$, assume it obtains its maximal at $x_{0}\in S^{n}$.
Then critical equation is
\begin{equation} \label{a.8}
(\frac{1}{2}|\nabla\gamma|^{2})_{i}=\sum_{k}\gamma_{k}\gamma_{ki}=0\qquad \forall i
\end{equation}
\par
At $x_{0}$, by proper choice of orthogonal frame, we could assume $\{\gamma_{ij}\}$ is diagonal.
\par
Then at $x_{0}$, we have
\begin{eqnarray} \label{a.9}
0&\geq & \sum_{i,j}\widetilde{F^{ij}}(\frac{1}{2}|\nabla\gamma|^{2})_{ij}=\sum_{i,j,k}(\widetilde{F^{ij}}\gamma_{ki}\gamma_{kj}+
\widetilde{F^{ij}}\gamma_{k}\gamma_{kij})\notag\\
&=&\sum_{i}\widetilde{F^{ii}}\gamma_{ii}^{2}+\sum_{i,j,k}\gamma_{k}\widetilde{F^{ij}}(\gamma_{ijk}-\delta_{ik}\gamma_{j}+\delta_{ij}\gamma_{k})\\
&=&\sum_{i}\widetilde{F^{ii}}\gamma_{ii}^{2}+\sum_{i,j}\widetilde{F^{ij}}\delta_{ij}\cdot |\nabla\gamma|^{2}-\sum_{i,j}\widetilde{F^{ij}}\gamma_{i}\gamma_{j}
+\sum_{i,j,k}\gamma_{k}\widetilde{F^{ij}}\gamma_{ijk}\notag
\end{eqnarray}

\par
From now on till the end of theorem \ref{T3}'s proof, $\sigma_{k}$ means $\sigma_{k}(b_{j}^{i})$.

Differentiate with equation (\ref{a.7}), using critical equation (\ref{a.8}), we have
\begin{equation} \label{a.10}
f_{\theta}=\frac{(n-k)\phi^{n-k-1}\phi^{'}\rho_{\theta}}{\omega^{k-1}}\cdot\sigma_{k}+\sum_{\alpha,\beta}\frac{\phi^{n-k}}{\omega^{k-1}}
\frac{\partial\sigma_{k}}{\partial b_{\beta}^{\alpha}}(b_{\beta}^{\alpha})_{\theta}
\end{equation}
\par
Once again, using critical equation (\ref{a.8})
\begin{eqnarray*}
& &\sum_{\theta}\gamma_{\theta}f_{\theta}=\frac{(n-k)\phi^{n-k}\phi^{'}|\nabla\gamma|^{2}}{\omega^{k-1}}\cdot\sigma_{k}\\
& &+\sum_{\alpha,\beta,\theta,s,t}\frac{\phi^{n-k}}{\omega^{k-1}}\cdot\gamma_{\theta}\cdot
\frac{\partial\sigma_{k}}{\partial b_{\beta}^{\alpha}}(\delta^{\alpha s}-\frac{\gamma_{\alpha}\gamma_{s}}{\omega(\omega+1)})(\delta^{t \beta}-\frac{\gamma_{t}\gamma_{\beta}}{\omega(\omega+1)})(-\gamma_{st\theta}+\phi^{''}\rho_{\theta}\gamma_{s}
\gamma_{t}+\phi^{''}\rho_{\theta}\delta_{st})\\
& &=\frac{(n-k)\phi^{n-k}\phi^{'}|\nabla\gamma|^{2}}{\omega^{k-1}}\cdot\sigma_{k}+
\frac{\phi^{n-k}}{\omega^{k-1}}\sum_{s,t}[\sum_{\theta}\gamma_{\theta}\widetilde{F^{st}}(-\gamma_{st\theta})+
(\widetilde{F^{st}}\gamma_{s}\gamma_{t}+\widetilde{F^{st}}\delta_{st})\phi^{''}\phi\cdot|\nabla\gamma|^{2}]
\end{eqnarray*}
\par
Then we obtain
\begin{equation} \label{a.15}
\sum_{s,t,\theta}\gamma_{\theta}\widetilde{F^{st}}\gamma_{st\theta}=(n-k)\phi^{'}\cdot|\nabla\gamma|^{2}\cdot\sigma_{k}
+\sum_{s,t}(\widetilde{F^{st}}\gamma_{s}\gamma_{t}+\widetilde{F^{st}}\delta_{st})\phi^{''}\phi\cdot|\nabla\gamma|^{2}-
\frac{\omega^{k-1}}{\phi^{n-k}}\sum_{\theta}\gamma_{\theta}f_{\theta}
\end{equation}
Combining (\ref{a.15}) with (\ref{a.9}), now we have
\begin{eqnarray*}
0&\geq & \sum_{i}\widetilde{F^{ii}}\gamma_{ii}^{2}+\sum_{i,j}\widetilde{F^{ij}}\delta_{ij}\cdot |\nabla\gamma|^{2}-\sum_{i,j}\widetilde{F^{ij}}\gamma_{i}\gamma_{j}\\
& &+\sum_{s,t}(\widetilde{F^{st}}\gamma_{s}\gamma_{t}+\widetilde{F^{st}}\delta_{st})\phi^{2}\cdot|\nabla\gamma|^{2}
+(n-k)\phi^{'}\cdot|\nabla\gamma|^{2}\cdot\sigma_{k}-\frac{\omega^{k-1}}{\phi^{n-k}}\sum_{\theta}\gamma_{\theta}f_{\theta}\\
&\geq &(n-k)\phi^{'}\cdot|\nabla\gamma|^{2}\cdot\sigma_{k}-\frac{\omega^{k-1}}{\phi^{n-k}}\sum_{\theta}\gamma_{\theta}f_{\theta}
=(n-k)\phi^{'}\cdot|\nabla\gamma|^{2}\cdot f\cdot\frac{\omega^{k-1}}{\phi^{n-k}}
-\frac{\omega^{k-1}}{\phi^{n-k}}\sum_{\theta}\gamma_{\theta}f_{\theta}\\
&\geq & \frac{\omega^{k-1}}{\phi^{n-k}}[(n-k)\phi^{'}\cdot|\nabla\gamma|^{2}\cdot f-|\nabla\gamma|\cdot|\nabla f|]
\end{eqnarray*}
Because $k<n$, we have $C^{1}$ estimate, i.e. $|\nabla\gamma|\leq C_{2}$, here
$C_{2}\sim inf(f)$, $||f||_{C^{1}}$, $n$, $k$.
\end{proof}

\section{$C^{2}$ estimate}

We now prove $C^{2}$ estimate. We first work on M and obtain its curvature estimate.
Therefore, in this section, all the covariant derivatives are with respect to the induced metric $g_{ij}$
on the hypersurface $M\in \mathbb{H}^{n+1}$.
\par
Choose local orthonormal frame $\{e_{1},\cdots,e_{n}\}$ on M. $\nu=e_{n+1}$ is the unit outer normal of hypersurface.
Let $\{h_{ij}\}$ be the second fundamental form with respect to this frame. Then under this frame,
the following identities hold.
\begin{equation} \label{3.1}
h_{ijk}=h_{ikj}
\end{equation}
\begin{equation} \label{3.2}
h_{iikk}=h_{kkii}+h_{ii}^{2}h_{kk}-h_{ii}h_{kk}^{2}-h_{ii}+h_{kk}
\end{equation}
Let $\lambda_{1},\cdots,\lambda_{n}$ be the principal curvatures of M. Suppose $\lambda_{1}\geq\lambda_{2}\cdots\geq\lambda_{n}$,
then the following equation defined on $S^{n}$
\begin{equation} \label{3.3}
\sigma_{k}(\kappa_{1},\cdots,\kappa_{n})=\sigma_{k}(\tilde{h_{j}^{i}})=\frac{f_{0}}{\phi^{n-1}\sqrt{\phi^{2}+|\nabla \rho|^{2}}}
\end{equation}
can be equivalently expressed as
\begin{equation} \label{3.4}
\sigma_{k}(\lambda_{1},\cdots,\lambda_{n})(X)=u(X)\cdot f(X) \qquad X\in M
\end{equation}
where $f_{0}>0$ is the given function on $S^{n}$. u is the support function.
Let $X=R_{M}(\theta)=(\theta,\rho(\theta))$, $\theta\in S^{n}$, then
\begin{equation} \label{3.5}
f(X)=\frac{f_{0}(\theta)}{\phi^{n+1}(\rho(\theta))}
\end{equation}
\par
Since we already have $C^{0}, C^{1}$ estimates for function $\rho$, it's easy to see for (\ref{3.5}), we have (Guan-Li-Li \cite{7})
\begin{equation} \label{3.6}
|f_{i}(X)|\leq C(n,k,inf(f_{0}),||f_{0}||_{C^{1}})\qquad \qquad \qquad\quad\forall i
\end{equation}
\begin{equation} \label{3.7}
|f_{ij}(X)|\leq C(n,k,inf(f_{0}),||f_{0}||_{C^{2}})(1+\lambda_{1})(X)\qquad \forall i,j
\end{equation}

\bigskip

\begin{theorem} \label{C2}
 If k-convex hypersurface M satisfies equation (\ref{3.4}) (or (\ref{3.3})) for some $1\leq k\leq n$,
then we have
\begin{eqnarray*}
 max_{M}\lambda_{1}\leq C
 \end{eqnarray*}
$C\sim n$, $k$, $\inf(f_{0})$, $||f_{0}||_{C^{2}}$, where $f_{0}$ is the given function on $S^{n}$
\end{theorem}

\begin{proof}
 We first prove the case $3\leq k\leq n$.
\par

We here derive $C^{2}$ estimate for equation (\ref{3.4}).

Now, let $\lambda_{1}$ denotes the biggest eigenvalue of M. Consider the following test function:
\begin{equation} \label{3.8}
\frac{\lambda_{1}g(\Phi)}{u-a}
\end{equation}
here $\phi(\rho)=\sinh(\rho)$, and $\Phi(\rho)=\int_{0}^{\rho}\phi(s)ds$, $g(\Phi)=e^{\beta\Phi},\beta$ to be chosen. $a=\frac{1}{N}\inf_{M}(u)$ with N large enough depending $n$, $k$, $\inf(f_{0})$, $||f_{0}||_{C^{1}}$ which will be determined later.

Assume this function obtains its maximum at point $x_{0}$. We use the trick from \cite{brendle}.
 Denote $P(x_{0})=\frac{\lambda_{1}g(\Phi)}{u-a}(x_{0})$. Define function $\psi$ satisfying
\begin{eqnarray*}
\frac{\psi(x)g(\Phi)(x)}{(u-a)(x)}=P(x_{0})
\end{eqnarray*}
It's easy to see that $\psi\geq\lambda_{1}$ and $\psi(x_{0})=\lambda_{1}(x_{0})$. Choose proper local orthonormal frame $e_{1},e_{2},\cdots,e_{n}$ around $x_{0}$, so that $\{h_{ij}\}$ is diagonal at $x_{0}$ and $h_{ii}=\lambda_{i}$, where $\lambda_{1}\geq\lambda_{2}\cdots\geq\lambda_{n}$.
\par
Let $\mu$ denote the multiplicity of the biggest curvature eigenvalue at $x_{0}$, such that $\lambda_{1}=\lambda_{2}=\cdots =\lambda_{\mu}>\lambda_{\mu+1}\geq\cdots\geq\lambda_{n}$. Then from Brendle-Choi-Daskalopoulos \cite{brendle} well-known result, we have the followings exist at $x_{0}$:
\begin{equation} \label{3.9}
\psi=\lambda_{1}\quad\quad\quad
\end{equation}
\begin{equation} \label{3.10}
h_{kli}=\psi_{i}\delta_{kl}\quad 1\leq k,l\leq\mu
\end{equation}
\begin{equation} \label{3.11}
\psi_{ii}\geq h_{11ii}+2\sum_{l>\mu}\frac{1}{\lambda_{1}-\lambda_{l}}h_{1li}^{2}
\end{equation}
\par
From now on, we only consider the case with no multiple roots of biggest eigenvalue, i.e $\mu=1$. As we shall see later, the proof of case $\mu>1$ is actually a special case of the proof of case $\mu=1$.
\par
All the calculations are happening at point $x_{0}$. We consider the test function:
\begin{equation} \label{3.12}
\frac{\psi(x) g(\Phi)(x)}{(u-a)(x)}
\end{equation}
\par
Since this test function has constant value, at $x_{0}$, $(\ln(\frac{\psi g(\Phi)}{(u-a)}))_{i}=0$, the critical equation is
\begin{equation} \label{3.15}
\frac{\psi_{i}}{\psi}+\frac{g^{'}\cdot\Phi_{i}}{g(\Phi)}-\frac{u_{i}}{u-a}=\frac{h_{11i}}{\lambda_{1}}+\beta\Phi_{i}-\frac{u_{i}}{u-a}=0
\end{equation}
\par
Also, at $x_{0}$, since $\{h_{ij}\}$ is diagonal, by Lemma 1, $\{\sigma_{k}^{ij}\}$ is also diagonal. Then we have
\begin{eqnarray*}
0  & \geq & \sum_{i}\sigma_{k}^{ii}(\ln\frac{\psi g(\Phi)}{(u-a)})_{ii}  \\
   & \geq & \sum_{i}\sigma_{k}^{ii}\frac{h_{11ii}+2\sum_{l\geq 2}\frac{1}{\lambda_{1}-\lambda_{l}}h_{1li}^{2}}{\lambda_{1}}-\sum_{i}\sigma_{k}^{ii}\frac{h_{11i}^{2}}{\lambda_{1}^{2}}
           +\sum_{i}\frac{g^{''}}{g} \sigma_{k}^{ii}\langle \frac{\partial}{\partial \rho},e_{i}\rangle^{2}\phi^{2}\\
   &  &        +\frac{g^{'}}{g}\sum_{i=1}\sigma_{k}^{ii}\cdot\phi^{'}        -\frac{g^{'}}{g}ku^{2}f-\sum_{i}\frac{(g^{'})^{2}}{g^{2}}\sigma_{k}^{ii}\langle \frac{\partial}{\partial \rho}, e_{i}\rangle^{2}\phi^{2}
-k\phi^{'}\frac{uf}{u-a}\\
   & &  +\sum_{i}\frac{u}{u-a}\sigma_{k}^{ii}h_{ii}^{2}
-\sum_{t}\frac{\phi\langle\frac{\partial}{\partial \rho},e_{t}\rangle(\sigma_{k})_{t}}{u-a}+\sum_{i}\frac{\phi^{2}}{(u-a)^{2}}\sigma_{k}^{ii}h_{ii}^{2}\langle\frac{\partial}{\partial \rho},e_{i}\rangle^{2}
\end{eqnarray*}

Here we already used the following basic results (details of proof in \cite{Junfang}):
\begin{equation} \label{3.16}
u_{i}=\phi h_{ii}\langle\frac{\partial}{\partial\rho},e_{i}\rangle
\end{equation}
\begin{equation} \label{3.17}
u_{ii}=\phi^{'}h_{ii}-uh_{ii}^{2}+\sum_{t}\phi\langle\frac{\partial}{\partial\rho},e_{t}\rangle h_{iit}
\end{equation}
\begin{equation} \label{3.18}
\Phi_{i}=\langle\phi\frac{\partial}{\partial\rho},e_{i}\rangle
\end{equation}
\begin{equation} \label{3.19}
\Phi_{ii}=\phi^{'}-h_{ii}u
\end{equation}
Again, by using (\ref{3.1}),(\ref{3.2}),(\ref{3.4}),(\ref{3.6}),(\ref{3.7}),(\ref{3.15}),(\ref{3.16}),(\ref{3.17}) and following equations
\begin{eqnarray*}
& & \sigma_{k}^{ii}h_{ii11}+\sigma_{k}^{ij,mt}h_{ij1}h_{mt1} = (uf)_{11} \\
& & \frac{h_{11t}\cdot f}{\lambda_{1}}-\frac{(\sigma_{k})_{t}}{u-a} = \frac{u_{t}f}{u-a}-\beta\Phi_{t}f-\frac{u_{t}f+uf_{t}}{u-a}
 = -\beta\Phi_{t}f-\frac{uf_{t}}{u-a}
\end{eqnarray*}
we have
\begin{eqnarray*}
0 &\geq & \underbrace{-\frac{\sum_{i,j,m,t}\sigma_{k}^{ij,mt}h_{ij1}h_{mt1}}{\lambda_{1}}+\sum_{i}2\sigma_{k}^{ii}\frac{\sum_{l\geq 2}\frac{1}{\lambda_{1}-\lambda_{l}}h_{1li}^{2}}{\lambda_{1}}}_{(A)}\underbrace{-\sum_{i}\sigma_{k}^{ii}\frac{h_{11i}^{2}}{\lambda_{1}^{2}}}_{(B)}\\
  & & +(k-1)uf\cdot h_{11}+(\beta\phi^{'}-1)\sum_{i\geq 1}\sigma_{k}^{ii}
+\sum_{i}\frac{a}{u-a}\sigma_{k}^{ii}h_{ii}^{2}+\sum_{i}\sigma_{k}^{ii}(\frac{u_{i}}{u-a})^{2}-C \quad (\clubsuit 1)
\end{eqnarray*}
Here and from now on, $C\sim n$, $k$, $\inf(f_{0})$, $||f_{0}||_{C^{2}}$.

\medskip

{\it Case 1: $|\lambda_{n}|\geq \varepsilon\cdot\lambda_{1}$}
\par
Here $\varepsilon$ is any sufficiently small fixed constant number to be chosen later.
\par
By using concavity of $\sigma_{k}^{1/k}$, i.e $\sum_{i,j}\frac{\partial\sigma_{k}^{1/k}(\lambda_{1},\lambda_{2}\cdots\lambda_{n})}{\partial\lambda_{i}\partial\lambda_{j}}\eta_{i}\eta_{j}\leq 0$
\par
A basic calculation yields
\begin{eqnarray*}
-\frac{\sum_{_{1\leq i,j,m,t\leq n}}\sigma_{k}^{ij,mt}h_{ij1}h_{mt1}}{\lambda_{1}}\geq-\frac{\sum_{i\neq j}\sigma_{k-2}(\lambda |ij)h_{ii1}h_{jj1}}{\lambda_{1}}\geq -C\frac{(\sigma_{k})_{1}^{2}}{\lambda_{1}}\geq -C\lambda_{1}
\end{eqnarray*}
Then $(\clubsuit 1)$ becomes
\begin{equation} \label{3.20}
0\geq -\sum_{i}\sigma_{k}^{ii}\frac{h_{11i}^{2}}{\lambda_{1}^{2}}-C\lambda_{1}+(\beta\phi^{'}-1)\sum_{i\geq 1}\sigma_{k}^{ii}+\sum_{i}\frac{a}{u-a}\sigma_{k}^{ii}h_{ii}^{2}+\sum_{i}\sigma_{k}^{ii}(\frac{u_{i}}{u-a})^{2}-C
\end{equation}
\par
From critical equation, $\frac{h_{11i}}{\lambda_{1}}=-\beta\Phi_{i}+\frac{u_{i}}{u-a}$, then $-\sigma_{k}^{ii}\frac{h_{11i}^{2}}{\lambda_{1}^{2}}\geq -(1+\varepsilon^{'})\sigma_{k}^{ii}(\frac{u_{i}}{u-a})^{2}-C(\varepsilon^{'})\sigma_{k}^{ii}$.
\par
By choosing $\varepsilon^{'}\sim n$, $k$, $\inf(f_{0})$, $||f_{0}||_{C^{1}}$ small enough so that
$\frac{1}{2}\frac{a}{u-a}\sigma_{k}^{ii}h_{ii}^{2}\geq\varepsilon^{'}\sigma_{k}^{ii}(\frac{u_{i}}{u-a})^{2}$.\\
Now equation (41) becomes
\begin{equation} \label{3.21}
0\geq -C\lambda_{1}-C\sum_{i\geq 1}\sigma_{k}^{ii}+c_{0}\sum_{i}\sigma_{k}^{ii}h_{ii}^{2}-C
\end{equation}
$c_{0}$ is a small but positive constant depending $n$, $k$, $\inf(f_{0})$, $||f_{0}||_{C^{1}}$.
\par
At this case, $\sigma_{k}^{nn}h_{nn}^{2}=\sigma_{k}^{nn}\lambda_{n}^{2}\geq\varepsilon^{2}\sigma_{k}^{nn}\lambda_{1}^{2}\gg \sigma_{k}^{ii}\quad \forall 1\leq i\leq n$.
Since $\sigma_{k}^{nn}\geq c(n)\sigma_{k-1}\geq c_{1}$, $c_{1}$ is a small but positive constant depending $n$, $k$, $\inf(f_{0})$, $||f_{0}||_{C^{1}}$.
This means $\sigma_{k}^{nn}\lambda_{n}^{2}\geq \varepsilon^{2}c_{1}\lambda_{1}^{2}\gg \lambda_{1}$.
Therefore, we have our estimate at this case.

\bigskip

We now separate the proof into several cases:
\par
$\emph{Case 2:}$ $\lambda_{2}\leq \frac{1}{C_{1}}\lambda_{1}$. $C_{1}$ is a sufficiently large constant to be chosen.

\medskip

$\emph{Case 3:}$ $\lambda_{2}\geq\frac{1}{C_{1}}\lambda_{1}$, $\lambda_{3}\leq\frac{1}{C_{2}}\lambda_{1}$. $C_{1}$ has been chosen in Case 1. $C_{2}$ is a sufficiently large constant depending n, k, $C_{1}$, $\inf(f_{0})$, $||f_{0}||_{C^{1}}$. $C_{2}$ is to be chosen.

\medskip

$\cdots\cdots\cdots$

\medskip

$\emph{Case $l+1$:}$ $\lambda_{2}\geq\frac{1}{C_{1}}\lambda_{1}$, $\lambda_{3}\geq\frac{1}{C_{2}}\lambda_{1}$,$\cdots$, $\lambda_{l}\geq\frac{1}{C_{l-1}}\lambda_{1}$, $\lambda_{l+1}\leq\frac{1}{C_{l}}\lambda_{1}$. $C_{1},C_{2},\cdots, C_{l-1}$ have been chosen in Case 1, $\cdots$,Case $l$. $C_{l}$ is a sufficiently large constant depending n, k, $C_{1},C_{2},\cdots, C_{l-1}$, $\inf(f_{0})$, $||f_{0}||_{C^{1}}$. $\quad$ $C_{l}$ is to be chosen.

\medskip

$\cdots$
\par
$\emph{Case $k-1$:}$
$\lambda_{2}\geq\frac{1}{C_{1}}\lambda_{1}$, $\lambda_{3}\geq\frac{1}{C_{2}}\lambda_{1}$,$\cdots$, $\lambda_{k-2}\geq\frac{1}{C_{k-3}}\lambda_{1}$,
 $\lambda_{k-1}\leq\frac{1}{C_{k-2}}\lambda_{1}$. $C_{1},C_{2},\cdots, C_{k-3}$ have been chosen in Case 1, $\cdots$,Case $k-2$. $C_{k-2}$ is a sufficiently large constant depending n, k, $C_{1},C_{2},\cdots, C_{k-3}$, $\inf(f_{0})$, $||f_{0}||_{C^{1}}$. $C_{k-2}$ is to be chosen.

\medskip

$\emph{Case $k$:}$ $\lambda_{2}\geq\frac{1}{C_{1}}\lambda_{1}$, $\lambda_{3}\geq\frac{1}{C_{2}}\lambda_{1}$,$\cdots$, $\lambda_{k-2}\geq\frac{1}{C_{k-3}}\lambda_{1}$, $\lambda_{k-1}\geq\frac{1}{C_{k-2}}\lambda_{1}$. $C_{1},C_{2},\cdots, C_{k-2}$ have been chosen in Case 1, $\cdots$,Case $k-1$.

We now prove Case $l+1$, $1\leq l\leq k-2 $.
\par
Back to $(\clubsuit 1)$, let $\lambda=(\lambda_{1},\cdots,\lambda_{n})$

\begin{eqnarray} \label{3.22}
(A)+(B) & \geq & \underbrace{\frac{\sum_{i\geq 2}(2\sigma_{k-2}(\lambda |i1)h_{11i}^{2}+\frac{2\sigma_{k}^{11}}{\lambda_{1}-\lambda_{i}}h_{11i}^{2}-\frac{\sigma_{k}^{ii}}{\lambda_{1}}h_{11i}^{2})}{\lambda_{1}}}_{(C_{1})}\notag\\
& & \underbrace{+\frac{-\sum_{i\neq j}\sigma_{k-2}(\lambda |ij)h_{ii1}h_{jj1}+\sum_{i\geq 2}\frac{2\sigma_{k}^{ii}}{\lambda_{1}-\lambda_{i}}h_{ii1}^{2}}{\lambda_{1}}}_{(C_{2})}
\underbrace{-\sigma_{k}^{11}(\frac{h_{111}}{\lambda_{_{1}}})^{2}}_{(C_{3})}
\end{eqnarray}

A easy calculation shows
\begin{equation} \label{3.23}
(C_{1})=\frac{\sum_{i\geq 2}\frac{\lambda_{1}+\lambda_{i}}{(\lambda_{1}-\lambda_{i})\lambda_{1}}\sigma_{k}^{ii}h_{11i}^{2}}{\lambda_{1}}>0
\end{equation}
This is because we could always assume $|\lambda_{n}|\leq \lambda_{1}$, otherwise we could use Case 1 to prove.
\par
As for $(C_{3})$, using critical equation, same as the proof of Case 1, by letting $\lambda_{1}$ large enough, we have
\begin{equation} \label{3.24}
-\sigma_{k}^{11}(\frac{h_{111}}{\lambda_{_{1}}})^{2}+\frac{a}{u-a}\sigma_{k}^{11}h_{11}^{2}+\sigma_{k}^{11}(\frac{u_{1}}{u-a})^{2}
\geq -C\sigma_{k}^{11}+\frac{1}{2}\frac{a}{u-a}\sigma_{k}^{11}h_{11}^{2}>0
\end{equation}
\par
Therefore, combining (\ref{3.22}), (\ref{3.23}), (\ref{3.24}), $(\clubsuit 1)$ becomes
\begin{eqnarray*}
0\geq \underbrace{\frac{-\sum_{i\neq j}\sigma_{k-2}(\lambda |ij)h_{ii1}h_{jj1}+\sum_{i\geq 2}\frac{2\sigma_{k}^{ii}}{\lambda_{1}-\lambda_{i}}h_{ii1}^{2}}{\lambda_{1}}}_{(C_{2})} \\
+(k-1)\cdot uf\cdot h_{11}+(\beta\phi^{'}-1)\sum_{i\geq 2}\sigma_{k}^{ii}-C\quad\quad (\clubsuit 2)
\end{eqnarray*}
\par
We now begin to prove Case $l+1$:
\par
$\emph{Case $l+1$:}$ $\lambda_{2}\geq\frac{1}{C_{1}}\lambda_{1}$, $\lambda_{3}\geq\frac{1}{C_{2}}\lambda_{1}$,$\cdots$, $\lambda_{l}\geq\frac{1}{C_{l-1}}\lambda_{1}$, $\lambda_{l+1}\leq\frac{1}{C_{l}}\lambda_{1}$. $C_{1},C_{2},\cdots, C_{l-1}$ have been chosen in Case 1, $\cdots$,Case $l$. $C_{l}$ is a sufficiently large constant depending n, k, $C_{1},C_{2},\cdots, C_{l-1}$, $\inf(f_{0})$, $||f_{0}||_{C^{1}}$. $\quad$ $C_{l}$ is to be chosen.
\par
We first prove some basic facts under current case.

\medskip

$\mathbf{FACT}$ $\mathbf{1}$:$\qquad$ $\theta_{2}\lambda_{1}^{l}\geq\sigma_{l}\geq\theta_{1}\lambda_{1}^{l}\quad\quad \theta_{1},\theta_{2}\sim n,C_{1},\cdots, C_{l-1}$
\begin{proof} Since $(\lambda_{1},\cdots,\lambda_{n})\in\Gamma_{k}$, we have $\sigma_{k-1}(\lambda |1)>0$, $\sigma_{k-2}(\lambda |12)>0$,
$\cdots$, $\sigma_{1}(\lambda |1\cdots k-1)>0$.
This means $ \lambda_{k}+\lambda_{k+1}+\cdots +\lambda_{n}>0$.
\par
Suppose $\lambda_{1}>\lambda_{2}\geq\cdots\geq\lambda_{t}>0\geq\lambda_{t+1}\geq\cdots\geq\lambda_{n}$, apparently, $t\geq k$.\\
So
\begin{eqnarray*}
0 &< &\lambda_{k}+\lambda_{k+1}+\cdots +\lambda_{t}+\lambda_{t+1}\cdots+\lambda_{n} \\
  & \leq & (t-k+1)\lambda_{k}+\lambda_{t+1}\cdots+\lambda_{n} \\
  & \leq & \lambda_{n}+(t-k+1)\lambda_{k}\quad\quad \lambda_{j}\leq 0, j\geq t+1
\end{eqnarray*}
therefore
\begin{equation} \label{3.25}
 |\lambda_{n}|\leq c(n)\lambda_{k}\leq c(n)\lambda_{l+1}
\end{equation}
Using (\ref{3.25}), we have
\begin{eqnarray*}
 \sigma_{l} &=& \lambda_{1}\lambda_{2}\cdots\lambda_{l}  +\sum_{i_{1}<\cdots <i_{l}}\lambda_{i_{1}}\cdots\lambda_{i_{l}}\quad i_{l}\geq l+1 \\
            & & \geq \lambda_{1}\lambda_{2}\cdots\lambda_{l}-c(n)\lambda_{1}\lambda_{2}\cdots\lambda_{l-1}\lambda_{l+1}
\end{eqnarray*}
\par
Since $\lambda_{l}\geq \frac{1}{C_{l-1}}\lambda_{1}$, $\lambda_{l+1}\leq \frac{1}{C_{l}}\lambda_{1}$,
we simply let $C_{l}$ satisfying
\begin{eqnarray*}
C_{l}\geq 2c(n)\cdot C_{l-1} \qquad \qquad\qquad\qquad\quad (\spadesuit 1)
\end{eqnarray*}
Then we have
\begin{eqnarray*} \label{3.26}
 \sigma_{l}\geq\frac{1}{2}\lambda_{1}\lambda_{2}\cdots\lambda_{l}\geq \theta_{1}\lambda_{1}^{l},\quad \theta_{1}\sim n,C_{1},\cdots, C_{l-1}
\end{eqnarray*}
the upper bound is obvious.
\end{proof}

The following facts can be proved similarly as Fact 1:

\medskip

$\mathbf{\mathbf{FACT}}$ $\mathbf{2:}$ $\theta_{4}\lambda_{1}^{l-1}\geq\sigma_{l}^{nn}\geq\cdots\geq\sigma_{l}^{11}\geq\theta_{3}\lambda_{1}^{l-1}\quad\quad \theta_{3},\theta_{4}\sim n,C_{1},\cdots, C_{l-1}$

\medskip

$\mathbf{FACT}$ $\mathbf{3:}$ $\sigma_{l-2}(\lambda |ij)\leq\theta_{5}\cdot\lambda_{1}^{l-2}$

\medskip

Remark: $\theta_{2},\theta_{4},\theta_{5}$ only depend on n. $\theta_{1},\theta_{3}$ only depend on $n,C_{1},\cdots, C_{l-1}$.
No matter how big $C_{l}$ is, as long as it satisfying $C_{l}\geq 2c(n)C_{l-1}$, we will have these facts exist.

\medskip

Now we handle Term $(C_{2})$.
By using concavity of $(\frac{\sigma_{k}}{\sigma_{l}}(\lambda_{1},\cdots,\lambda_{n}))^{\frac{1}{k-l}}$, i.e
\begin{eqnarray*}
\sum_{i,j}(\frac{\sigma_{k}}{\sigma_{l}}(\lambda_{1},\cdots,\lambda_{n}))^{\frac{1}{k-l}}_{ij}\eta_{i}\eta_{j}\leq 0
\end{eqnarray*}
Using Lemma 5, we have
\begin{eqnarray*}
 -\sum_{i\neq j}\sigma_{k-2}(\lambda |ij)h_{ii1}h_{jj1} &\geq & (1+\frac{1}{k-l})\sigma_{k}\cdot (\frac{\sum_{i=1}^{n}\sigma_{l}^{ii}h_{ii1}}{\sigma_{l}})^{2}-\frac{2}{k-l}(\sigma_{k})_{1}\cdot \frac{\sum_{i=1}^{n}\sigma_{l}^{ii}h_{ii1}}{\sigma_{l}}\\
 & &-(1-\frac{1}{k-l})\frac{(\sigma_{k})_{1}^{2}}{\sigma_{k}}-\sigma_{k}\frac{\sum_{i\neq j}\sigma_{l-2}(\lambda |ij)h_{ii1}h_{jj1}}{\sigma_{l}}
\end{eqnarray*}
Therefore
\begin{eqnarray*}
-\sum_{i\neq j}\sigma_{k-2}(\lambda |ij)h_{ii1}h_{jj1} &\geq & \underbrace{(1+\frac{1}{k-l})\sigma_{k}\cdot (\frac{\sigma_{l}^{11}h_{111}}{\sigma_{l}})^{2}-\frac{2}{k-l}(\sigma_{k})_{1}\cdot \frac{\sigma_{l}^{11}h_{111}}{\sigma_{l}}}_{(D)}\\
& &\underbrace{-(1-\frac{1}{k-l})\frac{(\sigma_{k})_{1}^{2}}{\sigma_{k}}}_{(D)}
+\underbrace{2(1+\frac{1}{k-l})\sigma_{k}\cdot(\frac{\sigma_{l}^{11}h_{111}}{\sigma_{l}}) \cdot\frac{\sum_{i=2}^{l}\sigma_{l}^{ii}h_{ii1}}{\sigma_{l}}}_{(E11)}\\
& & \underbrace{-\frac{2}{k-l}(\sigma_{k})_{1}\cdot \frac{\sum_{i=2}^{l}\sigma_{l}^{ii}h_{ii1}}{\sigma_{l}}}_{(E12)}
\underbrace{-2\sigma_{k}\frac{\sum_{i=2}^{l}\sigma_{l-2}(\lambda |1i)h_{111}h_{ii1}}{\sigma_{l}}}_{(E13)}\\
& & \underbrace{+(1+\frac{1}{k-l})\sigma_{k}\cdot (\frac{\sum_{i=2}^{l}\sigma_{l}^{ii}h_{ii1}}{\sigma_{l}})^{2}-2\sigma_{k}\frac{\sum_{2\leq i<j\leq l}\sigma_{l-2}(\lambda |ij)h_{ii1}h_{jj1}}{\sigma_{l}}}_{(E2)}\\
& &\underbrace{+(1+\frac{1}{k-l})\sigma_{k}\cdot (\frac{\sum_{j\geq l+1}\sigma_{l}^{jj}h_{jj1}}{\sigma_{l}})^{2}
-\frac{2}{k-l}(\sigma_{k})_{1}\cdot \frac{\sum_{j\geq l+1}\sigma_{l}^{jj}h_{jj1}}{\sigma_{l}}}_{(F)}\\
& &\underbrace{-2\sigma_{k}\frac{\sum_{j\geq l+1}\sigma_{l-2}(\lambda |1j)h_{111}h_{jj1}}{\sigma_{l}}-2\sigma_{k}\frac{\sum_{2\leq i\leq l,j\geq l+1}\sigma_{l-2}(\lambda |ij)h_{ii1}h_{jj1}}{\sigma_{l}}}_{(F)}\\
& & +\underbrace{2(1+\frac{1}{k-l})\sigma_{k}\cdot(\frac{\sum_{i\geq 2}^{l}\sigma_{l}^{ii}h_{ii1}}{\sigma_{l}}) \cdot\frac{\sum_{j\geq l+1}\sigma_{l}^{jj}h_{jj1}}{\sigma_{l}}}_{(F)}\\
& & +\underbrace{2(1+\frac{1}{k-l})\sigma_{k}\cdot(\frac{\sigma_{l}^{11}h_{111}}{\sigma_{l}}) \cdot\frac{\sum_{j\geq l+1}\sigma_{l}^{jj}h_{jj1}}{\sigma_{l}}}_{(F)}\\
& & \underbrace{-2\sigma_{k}\frac{\sum_{l+1\leq m<j}\sigma_{l-2}(\lambda |mj)h_{mm1}h_{jj1}}{\sigma_{l}}}_{(F)}
\end{eqnarray*}

By Fact 1, Fact 2, Fact 3, it's easy for us to obtain
\begin{eqnarray*}
 (F) &\geq & -\varepsilon_{1}u_{1}^{2}-\varepsilon_{2}\cdot\sigma_{k}\cdot\sum_{i\geq 2}^{l}(\frac{\sigma_{l}^{ii}}{\sigma_{l}})^{2}h_{ii1}^{2}\\
& & -C(\varepsilon_{1},\varepsilon_{2},n,C_{1},\cdots,C_{l-1})\sum_{j\geq l+1}\frac{h_{jj1}^{2}}{\lambda_{1}^{2}}-C\qquad\qquad \mathbf{(\infty 1)}
\end{eqnarray*}
 $\varepsilon_{1},\varepsilon_{2}$ are small constants to be chosen. (Here we already used critical equation: $\frac{h_{111}}{\lambda_{1}}+\beta\Phi_{1}-\frac{u_{1}}{u-a}=0$)

The main part of this proof is how to handle term $(E11),(E12),(E13)$.
\par
We first handle term $(E12)$.
\par
By critical equation,
\begin{equation} \label{3.27}
\frac{h_{111}}{\lambda_{1}}=\frac{u_{1}f}{(u-a)f}-\beta\Phi_{1}=\frac{(\sigma_{k})_{1}}{(u-a)f}-\frac{uf_{1}}{(u-a)f}-\beta\Phi_{1}
\end{equation}
Denote $ \tau :=\frac{uf_{1}}{(u-a)f}+\beta\Phi_{1}$. From now on, $\tau$ is a constant size value which can be controlled.\\
\begin{eqnarray*}
(E12) &= &-\frac{2}{k-l}\frac{(u-a)f}{\lambda_{1}}h_{111}\sum_{i\geq 2}^{l}\frac{\sigma_{l}^{ii}h_{ii1}}{\sigma_{l}}+\tau\sum_{i\geq 2}^{l}\frac{\sigma_{l}^{ii}h_{ii1}}{\sigma_{l}} \\
      &=&-\frac{2}{k-l}(\frac{(u-a)f}{\sigma_{k}}\cdot\frac{\sigma_{l}}{\lambda_{1}\sigma_{l}^{11}}-1)\sigma_{k}\frac{\sigma_{l}^{11}h_{111}}{\sigma_{l}}\sum_{i\geq 2}^{l}\frac{\sigma_{l}^{ii}h_{ii1}}{\sigma_{l}}\\
      &  &-\frac{2}{k-l}\sigma_{k}\frac{\sigma_{l}^{11}h_{111}}{\sigma_{l}}\sum_{i\geq 2}^{l}\frac{\sigma_{l}^{ii}h_{ii1}}{\sigma_{l}}+\tau\sum_{i\geq 2}^{l}\frac{\sigma_{l}^{ii}h_{ii1}}{\sigma_{l}}\\
      &= &-\frac{2}{k-l}(\frac{(u-a)f}{\sigma_{k}}-1+1-\frac{\sigma_{l}^{11}\lambda_{1}}{\sigma_{l}})\frac{\sigma_{k}}{(u-a)f}(\sigma_{k})_{1}\sum_{i\geq 2}^{l}\frac{\sigma_{l}^{ii}h_{ii1}}{\sigma_{l}}\\
      &  &-\frac{2}{k-l}\sigma_{k}\frac{\sigma_{l}^{11}h_{111}}{\sigma_{l}}\sum_{i\geq 2}^{l}\frac{\sigma_{l}^{ii}h_{ii1}}{\sigma_{l}}+\tau\sum_{i\geq 2}^{l}\frac{\sigma_{l}^{ii}h_{ii1}}{\sigma_{l}}\\
      &= &-\frac{2}{k-l}(-\frac{a}{u}+1-\frac{\sigma_{l}^{11}\lambda_{1}}{\sigma_{l}})\frac{\sigma_{k}}{(u-a)f}(\sigma_{k})_{1}\sum_{i\geq 2}^{l}\frac{\sigma_{l}^{ii}h_{ii1}}{\sigma_{l}}\\
      &  &-\frac{2}{k-l}\sigma_{k}\frac{\sigma_{l}^{11}h_{111}}{\sigma_{l}}\sum_{i\geq 2}^{l}\frac{\sigma_{l}^{ii}h_{ii1}}{\sigma_{l}}+\tau\sum_{i\geq 2}^{l}\frac{\sigma_{l}^{ii}h_{ii1}}{\sigma_{l}}
\end{eqnarray*}

Before continue to handle $(E12)$, we prove the following Fact 4:
\par
$\mathbf{FACT}$ $\mathbf{4:}$ $\quad$ $1\geq\frac{\sigma_{l}^{11}\lambda_{1}}{\sigma_{l}}\geq 1-C(n,C_{1},\cdots,C_{l-1})\frac{1}{C_{l}}$
\begin{proof} At current case, $\frac{\sigma_{l}^{11}\lambda_{1}}{\sigma_{l}}=\frac{\sigma_{l}^{11}\lambda_{1}}{\sigma_{l}^{11}\lambda_{1}+\sigma_{l}(\lambda |1)}\leq 1$ (Since $l\leq k-2$). Besides, same as the proof of Fact 1:
\begin{eqnarray*}
\sigma_{l}(\lambda |1)\leq c(n)\lambda_{2}\cdots\lambda_{l}\cdot\lambda_{l+1}
\end{eqnarray*}
Therefore,
\begin{eqnarray*}
1-\frac{\sigma_{l}^{11}\lambda_{1}}{\sigma_{l}}=\frac{\sigma_{l}(\lambda |1)}{\sigma_{l}^{11}\lambda_{1}+\sigma_{l}(\lambda |1)}\leq\frac{\sigma_{l}(\lambda |1)}{\sigma_{l}^{11}\lambda_{1}}\\
\leq\frac{c(n)\lambda_{2}\cdots\lambda_{l}\cdot\lambda_{l+1}}{\theta_{3}\cdot\lambda_{1}^{l}}\leq \frac{c(n)}{\theta_{3}}\cdot\frac{1}{C_{l}}
\end{eqnarray*}
This implies,
\begin{eqnarray*}
1\geq\frac{\sigma_{l}^{11}\lambda_{1}}{\sigma_{l}}\geq 1-C(n,C_{1},\cdots,C_{l-1})\frac{1}{C_{l}}
\end{eqnarray*}
\end{proof}
\par
Define $\theta=\max\{\frac{1}{N},C(n,C_{1},\cdots,C_{l-1})\frac{1}{C_{l}}\}$, as we shall see, $\theta$ can be really small.
\par
Now we have
\begin{eqnarray} \label{3.28}
(E12) &\geq &-C\theta(\sigma_{k})_{1}\sum_{i\geq 2}^{l}\frac{\sigma_{l}^{ii}h_{ii1}}{\sigma_{l}}-\frac{2}{k-l}\sigma_{k}\frac{\sigma_{l}^{11}h_{111}}{\sigma_{l}}\sum_{i\geq 2}^{l}\frac{\sigma_{l}^{ii}h_{ii1}}{\sigma_{l}}+\tau\sum_{i\geq 2}^{l}\frac{\sigma_{l}^{ii}h_{ii1}}{\sigma_{l}}
\end{eqnarray}
\par
By using equation (\ref{3.28}),
\begin{align*}
&(E11)+(E12)+(E13)\geq
2\sigma_{k}\cdot(\frac{\sigma_{l}^{11}h_{111}}{\sigma_{l}}) \cdot\frac{\sum_{i=2}^{l}\sigma_{l}^{ii}h_{ii1}}{\sigma_{l}}\notag \\
&-2\sigma_{k}\frac{\sum_{i=2}^{l}\sigma_{l-2}(\lambda |1i)h_{111}h_{ii1}}{\sigma_{l}}-C\theta(\sigma_{k})_{1}\sum_{i\geq 2}^{l}\frac{\sigma_{l}^{ii}h_{ii1}}{\sigma_{l}}+\tau\sum_{i\geq 2}^{l}\frac{\sigma_{l}^{ii}h_{ii1}}{\sigma_{l}}  \\
&=\frac{2\sigma_{k}}{\sigma_{l}^{2}}\sum_{i=2}^{l}(\sigma_{l}^{11}\sigma_{l}^{ii}-\sigma_{l}\sigma_{l-2}(\lambda |1i))h_{111}h_{ii1}-C\theta(\sigma_{k})_{1}\sum_{i\geq 2}^{l}\frac{\sigma_{l}^{ii}h_{ii1}}{\sigma_{l}}+\tau\sum_{i\geq 2}^{l}\frac{\sigma_{l}^{ii}h_{ii1}}{\sigma_{l}}\notag\\
&=\frac{2\sigma_{k}}{\sigma_{l}^{2}}\sum_{i=2}^{l}(\sigma_{l-1}(\lambda |1i)^{2}-\sigma_{l}(\lambda |1i)\sigma_{l-2}(\lambda |1i))h_{111}h_{ii1}-C\theta(\sigma_{k})_{1}\sum_{i\geq 2}^{l}\frac{\sigma_{l}^{ii}h_{ii1}}{\sigma_{l}}+\tau\sum_{i\geq 2}^{l}\frac{\sigma_{l}^{ii}h_{ii1}}{\sigma_{l}}\notag
\end{align*}
the last equation we used Lemma 6.
\par
Using the same method in the proof of Fact 1, we obtain
\begin{eqnarray}\label{3.29}
|\sigma_{l-1}(\lambda |1i)| &\leq& c(n)\lambda_{1}\lambda_{2}\cdots\lambda_{l-2}\cdot\lambda_{l+1}\leq \frac{c(n)}{C_{l}}\cdot\lambda_{1}^{l-1}
\end{eqnarray}
\begin{eqnarray}\label{3.30}
|\sigma_{l}(\lambda |1i)|   &\leq& c(n)\lambda_{1}\lambda_{2}\cdots\lambda_{l-2}\cdot\lambda_{l+1}^{2}\leq\frac{c(n)}{C_{l}^{2}}\cdot\lambda_{1}^{l}
\end{eqnarray}
\par
Using Fact 1, Fact 2, Fact 3, (\ref{3.29}), (\ref{3.30}), we have
\begin{eqnarray*}
&(E11)+(E12)+(E13)\geq -C\frac{1}{\theta_{1}^{2}}\frac{1}{C_{l}^{2}}\frac{|h_{111}|}{\lambda_{1}}\sum_{i=2}^{l}\frac{|h_{ii1}|}{\lambda_{1}}
-C\theta(\sigma_{k})_{1}\sum_{i\geq 2}^{l}\frac{\sigma_{l}^{ii}h_{ii1}}{\sigma_{l}}+\tau\sum_{i\geq 2}^{l}\frac{\sigma_{l}^{ii}h_{ii1}}{\sigma_{l}}\\
&\geq -C\frac{1}{\theta_{1}^{2}}\frac{1}{C_{l}^{2}}\frac{|h_{111}|}{\lambda_{1}}\sum_{i=2}^{l}|\frac{\sigma_{l}^{ii}}{\sigma_{l}}h_{ii1}|
-C\theta(\sigma_{k})_{1}\sum_{i\geq 2}^{l}\frac{\sigma_{l}^{ii}h_{ii1}}{\sigma_{l}}+\tau\sum_{i\geq 2}^{l}\frac{\sigma_{l}^{ii}h_{ii1}}{\sigma_{l}}\\
&\geq -(\frac{C(n,C_{1},\cdots,C_{l-1})}{C_{l}^{2}}+C\theta)\lambda_{1}^{2}-(\frac{C(n,C_{1},\cdots,C_{l-1})}{C_{l}^{2}}+C\theta)\sum_{i=2}^{l}(\frac{\sigma_{l}^{ii}}{\sigma_{l}})^{2}(h_{ii1})^{2}-C
\qquad\quad \mathbf{(\infty 2)}
\end{eqnarray*}

The second inequality we used Fact 1 and Fact 2. The last inequality we used critical equation.
\par
Now to $(E2)$, the method is same as handling (E11)+(E12)+(E13)
\begin{eqnarray*}
(E2)&=&(1+\frac{1}{k-l})\sigma_{k}\cdot (\frac{\sum_{i=2}^{l}\sigma_{l}^{ii}h_{ii1}}{\sigma_{l}})^{2}-2\sigma_{k}\frac{\sum_{2\leq i<j\leq l}\sigma_{l-2}(\lambda |ij)h_{ii1}h_{jj1}}{\sigma_{l}}\\
&\geq& \sigma_{k}\cdot (\frac{\sum_{i=2}^{l}\sigma_{l}^{ii}h_{ii1}}{\sigma_{l}})^{2}-2\sigma_{k}\frac{\sum_{2\leq i<j\leq l}\sigma_{l-2}(\lambda |ij)h_{ii1}h_{jj1}}{\sigma_{l}}\\
& &=\sigma_{k}\sum_{i=2}^{l}(\frac{\sigma_{l}^{ii}}{\sigma_{l}})^{2}h_{ii1}^{2}+\frac{2\sigma_{k}}{\sigma_{l}^{2}}\sum_{2\leq i<j\leq l}(\sigma_{l}^{ii}\sigma_{l}^{jj}-\sigma_{l}\sigma_{l-2}(\lambda |ij))h_{ii1}h_{jj1}\\
& &=\sigma_{k}\sum_{i=2}^{l}(\frac{\sigma_{l}^{ii}}{\sigma_{l}})^{2}h_{ii1}^{2}+\frac{2\sigma_{k}}{\sigma_{l}^{2}}\sum_{2\leq i<j\leq l}(\sigma_{l-1}(\lambda |ij)^{2}-\sigma_{l}(\lambda |ij)\sigma_{l-2}(\lambda |ij))h_{ii1}h_{jj1}\\
& &\geq \sigma_{k}\sum_{i=2}^{l}(\frac{\sigma_{l}^{ii}}{\sigma_{l}})^{2}h_{ii1}^{2}-2\sigma_{k}\frac{C(n,C_{1},\cdots,C_{l-1})}{C_{l}^{2}}\sum_{2\leq i< j\leq l}
\frac{|h_{ii1}|}{\lambda_{1}}\frac{|h_{jj1}|}{\lambda_{1}}\\
& &\geq (1-\frac{C(n,C_{1},\cdots,C_{l-1})}{C_{l}})\sigma_{k}\sum_{i=2}^{l}(\frac{\sigma_{l}^{ii}}{\sigma_{l}})^{2}h_{ii1}^{2}\qquad \qquad\quad \mathbf{(\infty 3)}
\end{eqnarray*}
In the second last inequality, we used similar result like (\ref{3.29}), (\ref{3.30}).
In the last equality, we used Fact 1 and Fact 2 like before.
\par
Lastly, we handle term (D). Using critical equation $\frac{h_{111}}{\lambda_{1}}+\beta\Phi_{1}-\frac{u_{1}}{u-a}=0$, we have
\begin{eqnarray*}
(D)&\geq& (1+\frac{1}{k-l}) (\frac{\sigma_{l}^{11}\lambda_{1}}{\sigma_{l}})^{2}(\frac{\sigma_{k}}{(u-a)f})^{2}\cdot\frac{(\sigma_{k})_{1}^{2}}{\sigma_{k}}\\
& &-\frac{2}{k-l}\frac{\sigma_{l}^{11}\lambda_{1}}{\sigma_{l}}\cdot\frac{\sigma_{k}}{(u-a)f}\cdot\frac{(\sigma_{k})_{1}^{2}}{\sigma_{k}}
-(1-\frac{1}{k-l})\frac{(\sigma_{k})_{1}^{2}}{\sigma_{k}}-C\lambda_{1}\\
& =&\frac{2}{k-l}\frac{\sigma_{l}^{11}\lambda_{1}}{\sigma_{l}}\cdot\frac{\sigma_{k}}{(u-a)f}
(\frac{\sigma_{l}^{11}\lambda_{1}}{\sigma_{l}}\cdot\frac{\sigma_{k}}{(u-a)f}-1)\cdot\frac{(\sigma_{k})_{1}^{2}}{\sigma_{k}}\\
& &+(1-\frac{1}{k-l}) (\frac{\sigma_{l}^{11}\lambda_{1}}{\sigma_{l}}\frac{\sigma_{k}}{(u-a)f}+1)(\frac{\sigma_{l}^{11}\lambda_{1}}{\sigma_{l}}\frac{\sigma_{k}}{(u-a)f}-1)
\cdot\frac{(\sigma_{k})_{1}^{2}}{\sigma_{k}}-C\lambda_{1}\\
&\geq& -c(n)|\frac{\sigma_{l}^{11}\lambda_{1}}{\sigma_{l}}\frac{\sigma_{k}}{(u-a)f}-1|\frac{(\sigma_{k})_{1}^{2}}{\sigma_{k}}-C\lambda_{1}\\
&=&-c(n)\frac{\sigma_{l}^{11}\lambda_{1}}{\sigma_{l}}|\frac{\sigma_{k}}{(u-a)f}-1+1-\frac{\sigma_{l}}{\sigma_{l}^{11}\lambda_{1}}|
\frac{(\sigma_{k})_{1}^{2}}{\sigma_{k}}-C\lambda_{1}\\
&\geq&-c(n)\cdot\theta\cdot\frac{(\sigma_{k})_{1}^{2}}{\sigma_{k}}-C\lambda_{1}\geq -C\cdot\theta\cdot\lambda_{1}^{2}-C\lambda_{1}\qquad\qquad\qquad\qquad \mathbf{(\infty 4)}
\end{eqnarray*}
\par
Now combine $(\infty 1), (\infty 2), (\infty 3), (\infty 4)$. We choose $\varepsilon_{2}$ small enough, say 1/10, and $C_{l}\geq A_{1}C(n,C_{1},\cdots,C_{l-1})$, $N\geq A_{2}\qquad\qquad\quad (\spadesuit 2)$.
\par
Here $A_{1}$, $A_{2}$ are sufficiently large constant only depending on $n$, $k$, $\inf(f_{0})$, $||f_{0}||_{C^{1}}$,
then we have
\begin{equation} \label{3.31}
 -\sum_{i\neq j}\sigma_{k-2}(\lambda |ij)h_{ii1}h_{jj1}\geq -(C\cdot\theta+\varepsilon_{1})\cdot\lambda_{1}^{2}
-C(\varepsilon_{1},\varepsilon_{2},n,C_{1},\cdots,C_{l-1})\sum_{j\geq l+1}\frac{h_{jj1}^{2}}{\lambda_{1}^{2}}
\end{equation}
\par
Before we give a final proof of Case l+1, we first prove a Lemma:
\par
\begin{lemma}:  $\quad$ For any sufficiently large constant L, we could always find sufficiently large constant $C_{l}$, $2\leq l\leq k-1$. $C_{l}$ depends on $n$, $k$, $L$.
\par
          If $\lambda_{l}\leq\frac{1}{C_{l}}\lambda_{1}$, then we have $\sigma_{k}^{nn}\geq\cdots\geq\sigma_{k}^{ll}\geq\frac{L}{\lambda_{1}}$.

\begin{proof}
 Using Lemma \ref{lemma2}, since $\sigma_{k}^{11}\geq\frac{c_{0}}{\lambda_{1}}$, $c_{0}\sim n$, $k$, $inf(f_{0})$, $||f_{0}||_{C^{1}}$, we only need to prove $\sigma_{k}^{ll}\geq L\cdot\sigma_{k}^{11}$.
 \begin{eqnarray*}
       \frac{\sigma_{k}^{ll}}{\sigma_{k}^{11}}&=&\frac{\sigma_{k-1}(\lambda |l)}{\sigma_{k-1}(\lambda |1)}=\frac{\lambda_{1}\sigma_{k-2}(\lambda |1l)
       +\sigma_{k-1}(\lambda |1l)}{\lambda_{l}\sigma_{k-2}(\lambda |1l)+\sigma_{k-1}(\lambda |1l)}\\
       &=& \frac{(\lambda_{1}-\lambda_{l})\cdot \sigma_{k-2}(\lambda |1l)}{\lambda_{l}\sigma_{k-2}(\lambda |1l)+\sigma_{k-1}(\lambda |1l)}+1
 \end{eqnarray*}
 \par
       If $\sigma_{k-1}(\lambda |1l)<0$, then we have $\frac{\sigma_{k}^{ll}}{\sigma_{k}^{11}}\geq\frac{\lambda_{1}-\lambda_{l}}{\lambda_{l}}\geq C_{l}-1$. Let $C_{l}=L+1$, we have our result.
       \par
       If $\sigma_{k-1}(\lambda |1l)\geq 0$, then by Maclaurin inequality $\sigma_{k-1}(\lambda|1l)\leq \sigma_{k-2}(\lambda |1l)^{\frac{k-1}{k-2}}$, then
       \begin{eqnarray*}
       \frac{\sigma_{k}^{ll}}{\sigma_{k}^{11}}\geq \frac{(\lambda_{1}-\lambda_{l})\cdot \sigma_{k-2}(\lambda |1l)}{\lambda_{l}\sigma_{k-2}(\lambda |1l)+\sigma_{k-2}(\lambda |1l)^{\frac{k-1}{k-2}}}
       \end{eqnarray*}
       Since $k-2\geq l-1$, we have
       \begin{eqnarray*}
       \sigma_{k-2}(\lambda |1l)=\sum_{i_{1}<i_{2}\cdots<i_{k-2};\quad i_{k-2}> l}\lambda_{i_{1}}\cdots\lambda_{i_{k-2}}
       \end{eqnarray*}
       Same reason as before
       \begin{eqnarray*}
       \sigma_{k-2}(\lambda |1l)\leq c(n)\cdot\lambda_{1}^{k-3}\cdot\lambda_{l}
       \end{eqnarray*}
       Therefore
       \begin{eqnarray*}
       \frac{(\lambda_{1}-\lambda_{l})\cdot \sigma_{k-2}(\lambda |1l)}{\sigma_{k-2}(\lambda |1l)^{\frac{k-1}{k-2}}}
       &=&(\frac{(\lambda_{1}-\lambda_{l})^{k-2}}{\sigma_{k-2}(\lambda |1l)})^{\frac{1}{k-2}}\\
       &\geq& (1-\frac{1}{C_{l}})(\frac{\lambda_{1}^{k-2}}{\sigma_{k-2}(\lambda |1l)})^{\frac{1}{k-2}}
       \geq (1-\frac{1}{C_{l}})(\frac{1}{c(n)})^{\frac{1}{k-2}}(\frac{\lambda_{1}}{\lambda_{l}})^{\frac{1}{k-2}}\\
       &\geq& (1-\frac{1}{C_{l}})(\frac{1}{c(n)})^{\frac{1}{k-2}}C_{l}^{\frac{1}{k-2}}\\
       \end{eqnarray*}
       Choosing $C_{l}$, large enough, depending $n$, $k$, $L$, we have our result.
       \end{proof}
       \end{lemma}
\par
Now we deal with Case l+1, ($l\leq k-2$)
\par
Firstly, we choose $\beta=2$, this implies $(\beta\phi^{'}-1)\sum_{i=1}^{n}\sigma_{k}^{ii}\geq \sum_{i=1}^{n}\sigma_{k}^{ii}$\\

Using (\ref{3.31}), $(\clubsuit 2)$ becomes
\begin{eqnarray} \label{888}
 0 &\geq & -(c(n)\cdot\theta+\varepsilon_{1})\cdot\lambda_{1}+(k-1)\cdot uf\cdot\lambda_{1}-C\\
& &+\frac{-C(\varepsilon_{1},\varepsilon_{2},n,C_{1},\cdots, C_{l-1})\sum_{j\geq l+1}\frac{h_{jj1}^{2}}{\lambda_{1}^{2}}
+\sum_{j\geq l+1}\frac{2\sigma_{k}^{jj}}{\lambda_{1}-\lambda_{j}}h_{jj1}^{2}}{\lambda_{1}}\notag
\end{eqnarray}
here $\theta=\max\{\frac{1}{N},\frac{C(n,C_{1},\cdots, C_{l-1})}{C_{l}}\}$.
\par
Firstly, we choose $N$, $C_{l}$ large enough such that
\begin{eqnarray*}
 C_{l}\geq A_{3}C(n,C_{1},\cdots, C_{l-1}),\quad N\geq A_{4},\quad
A_{3},A_{4}\sim n, k, \inf(f_{0}), ||f_{0}||_{C^{1}} \qquad (\spadesuit 3)
\end{eqnarray*}
Then choose $\varepsilon_{1}$ small enough so that
\begin{equation} \label{3.32}
-(c(n)\cdot\theta+\varepsilon_{1})\cdot\lambda_{1}+(k-1)\cdot uf\cdot\lambda_{1}\geq \frac{k-1}{2}\cdot uf\cdot\lambda_{1}
\end{equation}
\par
Now, since $\varepsilon_{1}$, $\varepsilon_{2}$, $C_{1},\cdots, C_{l-1}$ have been chosen, and we also have
\begin{eqnarray*}
\sum_{j\geq l+1}\frac{2\sigma_{k}^{jj}}{\lambda_{1}-\lambda_{j}}h_{jj1}^{2}\geq c(n)\sum_{j\geq l+1}\frac{\sigma_{k}^{jj}}{\lambda_{1}}h_{jj1}^{2}
\end{eqnarray*}
Using Lemma 7, we could choose $C_{l}$ large enough, satisfying
\begin{eqnarray*}
 C_{l}\geq A_{5}C(n,C_{1},\cdots, C_{l-1}),\quad A_{5}\sim n, k, \inf(f_{0}), ||f_{0}||_{C^{1}} \qquad\qquad (\spadesuit 4)
 \end{eqnarray*}
So that we have
\begin{equation} \label{3.33}
 \sigma_{k}^{jj}\geq\frac{C(\varepsilon_{1},\varepsilon_{2},n,C_{1},\cdots, C_{l-1})}{c(n)}\frac{1}{\lambda_{1}}\quad (j\geq l+1)
 \end{equation}
\par
Now combing $(\spadesuit 1)$, $(\spadesuit 2)$, $(\spadesuit 3)$, $(\spadesuit 4)$, we finally decide:
\begin{eqnarray*}
 C_{l}&=&\max\{ 2c(n)C_{l-1},\max\{A_{1},A_{3},A_{5}\}C(n,C_{1},\cdots, C_{l-1})\}\\
 N&=&\max\{A_{2},A_{4}\}
 \end{eqnarray*}
Using (\ref{3.32}), (\ref{3.33}), then (\ref{888}) becomes
\begin{eqnarray*}
0\geq\frac{k-1}{2}\cdot uf\cdot\lambda_{1}-C
\end{eqnarray*}
\par
We now finish the proof of Case l+1 $\qquad l\leq k-2$.

\bigskip

{\it Remark}: Discussion of $N,C_{1},\cdots,C_{k-2}$.
\par
From $(\spadesuit 2)$, $N\geq A_{2}$. $A_{2}$ can be written in the explicit form:
\begin{equation} \label{3.34}
N\geq A_{2}=c(n)\frac{1}{\inf(uf)}
\end{equation}
\par
From $(\spadesuit 3)$, $N\geq A_{4}$. $A_{4}$ can be written as
\begin{equation} \label{3.35}
A_{4}=c_{1}(n)\cdot\frac{\max(\phi^{2})\cdot \max(f)}{\inf(uf)\cdot \inf(u)}
\end{equation}
Let
\begin{equation} \label{3.36}
N=\max\{A_{2},A_{4}\}=\max\{c(n)\frac{1}{\inf(uf)},c_{1}(n)\cdot\frac{\max(\phi^{2})\cdot \max(f)}{\inf(uf)\cdot \inf(u)}\}
\end{equation}
then N only depends on $n$, $|f_{0}|_{C^{0}}$, $\inf(\rho)$, $|\rho|_{C^{0}}$, $|\nabla\rho|_{C^{0}}$, $\inf(f_{0})$.
\par
Since we already have $C^{0}$, $C^{1}$ estimate for $\rho$, this means
\begin{eqnarray*}
N\sim n,k, \inf(f_{0}), ||f_{0}||_{C^{1}}
\end{eqnarray*}
where $f_{0}$ is the given function on $S^{n}$.
\par
For $C_{1}$ in the Case 2:
\par
at this case, $(\clubsuit 1)$ becomes
\begin{eqnarray}\label{88888}
0 &\geq & -c_{2}(n)\cdot \max\{\frac{1}{N},\frac{c_{3}(n)}{C_{1}}\}\frac{\phi^{2}f^{2}}{\sigma_{k}}\cdot\lambda_{1}
-\varepsilon_{1}\cdot \phi^{2}f^{2}\cdot\lambda_{1}\notag \\
& &+(k-1)uf\cdot\lambda_{1}-C+\frac{-\frac{c_{4}(n)}{\varepsilon_{1}}\sum_{i\geq 2}(\frac{h_{ii1}}{\lambda_{1}})^{2}
+\sum_{i\geq 2}\frac{2\sigma_{k}^{ii}}{\lambda_{1}-\lambda_{i}}h_{ii1}^{2}}{\lambda_{1}}
\end{eqnarray}

First choose
\begin{equation} \label{3.37}
C_{1}\geq c_{5}(n)\cdot\frac{\max(\phi^{2})\max(f^{2})}{\min(uf)^{2}}
\end{equation}
\begin{equation} \label{3.38}
\varepsilon_{1}=\frac{1}{4}\cdot\frac{\min(uf)}{\max(\phi^{2})\max(f^{2})}
\end{equation}
Since we already chose appropriate N, then (\ref{88888}) becomes
\begin{equation} \label{3.39}
0\geq \frac{1}{2}uf\cdot\lambda_{1}-C+\frac{-\frac{c_{4}(n)}{\varepsilon_{1}}\sum_{i\geq 2}(\frac{h_{ii1}}{\lambda_{1}})^{2}
+\sum_{i\geq 2}\frac{2\sigma_{k}^{ii}}{\lambda_{1}-\lambda_{i}}h_{ii1}^{2}}{\lambda_{1}}
\end{equation}
\par
All c(n) and $c_{i}(n),\quad 1\leq i\leq 5$ are constants only depending on n.
\par
Now again, let $C_{1}$ large enough, so that the last term of (\ref{3.39}) is bigger than 0.
From Lemma 7, $C_{1}$ only depends on n,k,$\frac{c_{4}(n)}{\varepsilon_{1}}.\qquad\qquad\qquad\qquad\qquad (**)$
\par
Combining (\ref{3.37}), (\ref{3.38}), $(**)$, $C_{1}$ is a fixed constant only depending on $n,k,\inf(f_{0}),||f_{0}||_{C^{1}}$.
\par
Now by induction and $(\spadesuit 1),(\spadesuit 2),(\spadesuit 3), (\spadesuit 4)$, $C_{l}$ only depends on
\begin{eqnarray*}
n,k,C_{1},\cdots,C_{l-1},\inf(f_{0}),||f_{0}||_{C^{1}}
\end{eqnarray*}
This implies $C_{l}$ only depends on $n,k,\inf(f_{0}),||f_{0}||_{C^{1}}$.
\par
Therefore, we have chosen $N, C_{1},\cdots, C_{k-2}$. All of them are large, fixed constants only depending on
$n$, $k$, $\inf(f_{0})$, $||f_{0}||_{C^{1}}$.

\medskip

Now we already proved Case 1, Case 2,$\cdots$, Case k-1, the only remaining case is Case k:
\begin{eqnarray*}
\lambda_{2}\geq\frac{1}{C_{1}}\lambda_{1}, \lambda_{3}\geq\frac{1}{C_{2}}\lambda_{1},\cdots, \lambda_{k-2}\geq\frac{1}{C_{k-3}}\lambda_{1}, \lambda_{k-1}\geq\frac{1}{C_{k-2}}\lambda_{1},
\end{eqnarray*}
 $C_{1},C_{2},\cdots, C_{k-2}$ have been chosen in Case 2, $\cdots$,Case $k-1$. By the previous remark, they only depend on $n$, $k$, $\inf(f_{0})$, $||f_{0}||_{C^{1}}$.
\par
We need another Lemma:
\begin{lemma} At this case, we have $\sigma_{k}^{nn}=\sigma_{k-1}(\lambda |n)\geq c_{\alpha}\cdot\lambda_{1}^{k-1},\qquad c_{\alpha}\sim C_{1},\cdots ,C_{k-2}$
\begin{proof} Suppose $\lambda_{1}>\lambda_{2}\geq\cdots\geq\lambda_{t}>0\geq \lambda_{t+1}\geq\cdots\geq\lambda_{n}$. As we knew, $t\geq k$,\\
Therefore,
\begin{eqnarray}\label{3.40}
\sigma_{k}^{nn}&=& \sigma_{k-1}(\lambda |n)\geq \lambda_{1}\cdots\lambda_{k-1}+\sum_{i_{1}<\cdots <i_{k-1}\quad i_{k-1}\geq t+1}
\lambda_{i_{1}}\cdots\lambda_{i_{k-1}}\notag \\
\sigma_{k}^{nn}&=&\sigma_{k-1}(\lambda |n)\geq \lambda_{1}\cdots\lambda_{k-1}-c(n)\cdot\lambda_{1}^{k-2}\cdot |\lambda_{n}|\notag\\
& \geq&  \frac{1}{C_{1}\cdots C_{k-2}}\lambda_{1}^{k-1}-c(n)\cdot\lambda_{1}^{k-2}\cdot |\lambda_{n}|
\end{eqnarray}
The first inequality is due to we could always assume $|\lambda_{i}|<\lambda_{1}$, otherwise, we would use Case 1 to handle.
\par
Since $C_{1},\cdots, C_{k-2}$ have been chosen,\\
if $|\lambda_{n}|\geq\frac{1}{2C_{1}\cdots C_{k-2}}\cdot\frac{1}{c(n)}\cdot\lambda_{1}$, we could use Case 1 to handle.\\
if $|\lambda_{n}|\leq\frac{1}{2C_{1}\cdots C_{k-2}}\cdot\frac{1}{c(n)}\cdot\lambda_{1}$, from (\ref{3.40}) we have
\begin{eqnarray*}
\sigma_{k}^{nn}\geq \frac{1}{2C_{1}\cdots C_{k-2}}\lambda_{1}^{k-1}
\end{eqnarray*}
let $c_{\alpha}= \frac{1}{2C_{1}\cdots C_{k-2}}$, we prove Lemma 8.
\end{proof}
\end{lemma}

\medskip

From the proof of Case 1 and choose of $\beta$, we have
\begin{eqnarray*}
-\frac{\sum_{i\neq j}\sigma_{k-2}(\lambda |ij)h_{ii1}h_{jj1}}{\lambda_{1}}&\geq &-C\lambda_{1}-C\qquad C\sim n, k, \inf(f_{0}), ||f_{0}||_{C^{1}}\\
\beta\phi^{'}-1&\geq & 1
\end{eqnarray*}
Therefore,
$(\clubsuit 2)$ becomes
\begin{eqnarray*}
0 &\geq & -C\lambda_{1}-C+\sum_{i=2}^{n}\sigma_{k}^{ii}\\
  &\geq & -C\lambda_{1}-C+\sigma_{k}^{nn}\\
  &\geq & -C\lambda_{1}-C+c_{\alpha}\cdot\lambda_{1}^{k-1}
\end{eqnarray*}
We certainly have our estimate.

\medskip

{\it Remark:}  We actually don't need to prove so many cases. In fact, if k is odd, Case 1, Case 2, $\cdots$, Case $[\frac{k}{2}]+1$ are totally enough for our proof. And if k is even, Case 1, Case 2, $\cdots$, Case $\frac{k}{2}$ are enough for the proof.
Cause if Case 2, $\cdots$, (Case $[\frac{k}{2}]+1$ or Case $\frac{k}{2}$) all fail, then the same method of Lemma 8 shows $\sigma_{k}^{nn}\geq c_{\alpha}\lambda_{1}$, by choosing $\beta$ large enough, we still have our result.

\medskip

There finishes the proof of $C^{2}$ estimate when no multiple roots for biggest eigenvalue.

\bigskip

We now handle the case with multiple roots for biggest eigenvalue. ($\mu>1$)

\medskip

In this case, $(\clubsuit 1)$ becomes
\begin{eqnarray*}
0 &\geq & \underbrace{-\frac{\sum_{i,j,m,t}\sigma_{k}^{ij,mt}h_{ij1}h_{mt1}}{\lambda_{1}}+\sum_{i}2\sigma_{k}^{ii}\frac{\sum_{l\geq \mu+1}\frac{1}{\lambda_{1}-\lambda_{l}}h_{1li}^{2}}{\lambda_{1}}}_{(A)}\underbrace{-\sum_{i}\sigma_{k}^{ii}\frac{h_{11i}^{2}}{\lambda_{1}^{2}}}_{(B)}\\
  & & +(k-1)uf\cdot h_{11}+(\beta\phi^{'}-1)\sum_{i\geq 1}\sigma_{k}^{ii}
+\sum_{i}\frac{a}{u-a}\sigma_{k}^{ii}h_{ii}^{2}+\sum_{i}\sigma_{k}^{ii}(\frac{u_{i}}{u-a})^{2}-C \quad (\clubsuit 3)
\end{eqnarray*}

\begin{eqnarray} \label{new}
(A)+(B) & \geq & \underbrace{\frac{\sum_{i\geq \mu+1}(2\sigma_{k-2}(\lambda |i1)h_{11i}^{2}+\frac{2\sigma_{k}^{11}}{\lambda_{1}-\lambda_{i}}h_{11i}^{2}-\frac{\sigma_{k}^{ii}}{\lambda_{1}}h_{11i}^{2})}{\lambda_{1}}}_{(C_{1})}\notag\\
& & \underbrace{+\frac{-\sum_{i\neq j}\sigma_{k-2}(\lambda |ij)h_{ii1}h_{jj1}+\sum_{i\geq \mu+1}\frac{2\sigma_{k}^{ii}}{\lambda_{1}-\lambda_{i}}h_{ii1}^{2}}{\lambda_{1}}}_{(C_{2})}
\underbrace{-\sum_{i=1}^{\mu}\sigma_{k}^{ii}(\frac{h_{11i}}{\lambda_{_{1}}})^{2}}_{(C_{3})}
\end{eqnarray}
\par
Same as before, we could simply prove Case 1. This means we could assume $|\lambda_{n}|<\lambda_{1}$. Then $(C_{1})$ can be handled like before:
\begin{equation} \label{3.50}
(C_{1})=\frac{\sum_{i\geq \mu +1}\frac{\lambda_{1}+\lambda_{i}}{(\lambda_{1}-\lambda_{i})\lambda_{1}}\sigma_{k}^{ii}h_{11i}^{2}}{\lambda_{1}}\geq 0
\end{equation}
\par
As for $C_{3}$, we use critical equation: when $1\leq i\leq \mu$
\begin{equation} \label{3.51}
-\sigma_{k}^{ii}(\frac{h_{11i}}{\lambda_{1}})^{2}+\frac{a}{u-a}\sigma_{k}^{ii}h_{ii}^{2}+\sigma_{k}^{ii}(\frac{u_{i}}{u-a})^{2}
\geq -C\sigma_{k}^{ii}+\frac{1}{2}\frac{a}{u-a}\sigma_{k}^{ii}h_{ii}^{2}
\end{equation}
Since when $1\leq i\leq \mu$, $\lambda_{i}=\lambda_{1}$, we have this term is greater than zero.
\par
Therefore, using (\ref{3.50}), (\ref{3.51}), $(\clubsuit 3)$ becomes
\begin{eqnarray*}
0 &\geq &\frac{-\sum_{i\neq j}\sigma_{k-2}(\lambda |ij)h_{ii1}h_{jj1}+\sum_{i\geq \mu+1}\frac{2\sigma_{k}^{ii}}{\lambda_{1}-\lambda_{i}}h_{ii1}^{2}}{\lambda_{1}}\\
  & &+(k-1)\cdot uf\cdot h_{11}+(\beta\phi^{'}-1)\sum_{i\geq \mu+1}\sigma_{k}^{ii}-C\quad\quad (\clubsuit 4)
\end{eqnarray*}
\par
Since we already have $\lambda_{1}=\lambda_{2}=\cdots=\lambda_{\mu}$,
we only need to start from Case $\mu+1$,

\medskip

Case $\mu+1$: $\lambda_{1}=\lambda_{2}=\cdots=\lambda_{\mu}$, $\lambda_{\mu+1}\leq\frac{1}{C_{\mu}}\lambda_{1}, \quad C_{\mu}$ is large enough which will be determined later.

\medskip

Case $\mu+2$: $\lambda_{1}=\lambda_{2}=\cdots=\lambda_{\mu}$, $\lambda_{\mu+1}\geq\frac{1}{C_{\mu}}\lambda_{1}$, $\lambda_{\mu+2}\leq\frac{1}{C_{\mu+1}}\lambda_{1}$
$\quad C_{\mu}$ has been decided in Case $\mu+1$.  $\quad C_{\mu+1}$ is large enough which will be determined later.

\medskip

$\cdots\cdots$

\medskip

$\emph{Case $k-1$:}$
$\lambda_{2}=\lambda_{1}$,$\cdots$, $\lambda_{\mu}=\lambda_{1}$, $\lambda_{\mu+1}\geq\frac{1}{C_{\mu}}\lambda_{1}$, $\lambda_{\mu+2}\geq\frac{1}{C_{\mu+1}}\lambda_{_{1}}$,$\cdots$,$\lambda_{k-2}\geq\frac{1}{C_{k-3}}\lambda_{1}$, $\lambda_{k-1}\leq\frac{1}{C_{k-2}}\lambda_{1}$. $C_{\mu},C_{\mu+1},\cdots, C_{k-3}$ have been chosen in Case $\mu+1$, $\cdots$,Case $k-2$. $C_{k-2}$ is a sufficiently large constant depends on n, k, $C_{\mu},C_{\mu+1},\cdots, C_{k-3}$, $\inf(f_{0})$, $||f_{0}||_{C^{1}}$.$\quad$ $C_{k-2}$ is to be chosen.

\medskip

$\emph{Case $k$:}$
 $\lambda_{2}=\lambda_{1}$,$\cdots$, $\lambda_{\mu}=\lambda_{1}$, $\lambda_{\mu+1}\geq\frac{1}{C_{\mu}}\lambda_{1}$, $\lambda_{\mu+2}\geq\frac{1}{C_{\mu+1}}\lambda_{1}$,$\cdots$,$\lambda_{k-2}\geq\frac{1}{C_{k-3}}\lambda_{1}$, $\lambda_{k-1}\geq\frac{1}{C_{k-2}}\lambda_{1}$. $C_{\mu},C_{\mu+1},\cdots, C_{k-2}$ have been chosen in Case $\mu+1$, $\cdots$, Case $k-1$.

\medskip

When $\mu\leq k-1$, using exactly the same proof as before, we could solve Case $\mu+1$, $\cdots$, Case k to get $C^{2}$ estimate.
When $n>\mu\geq k-1$, we simply use $(\clubsuit 3)$, $(\clubsuit 4)$ and the same methods in solving Lemma 8 and Case k to obtain our result.
When $\mu=n$, we have $\lambda_{1}=\lambda_{2}=\cdots =\lambda_{n}$. It's obvious we have $C^{2}$ bound at this case.
\par
We finish the proof of $C^{2}$ estimate.

\end{proof}

\medskip

Now we prove case $k=1,2$. $k=2$ has been solved by Spruck-Xiao \cite{spruckX} and $k=1$ is trival. But for completeness, we still give their proofs.
\par
Same as before, we have
\begin{eqnarray*}
0&\geq& \frac{-\sum_{i\neq j}\sigma_{k-2}(\lambda |ij)h_{ii1}h_{jj1}+\sum_{i\geq \mu+1}\frac{2\sigma_{k}^{ii}}{\lambda_{1}-\lambda_{i}}h_{ii1}^{2}}{\lambda_{1}}\\
& &+(k-1)\cdot uf\cdot h_{11}+(\beta\phi^{'}-1)\sum_{i\geq \mu+1}\sigma_{k}^{ii}+
\frac{1}{4}\frac{a}{u-a}\sum_{i}\sigma_{k}^{ii}h_{ii}^{2}-C\quad\quad  (\Re)
\end{eqnarray*}
\par
$\mathbf{k=2}$: Since we have $\sum_{i\geq \mu+1}\sigma_{2}^{ii}\geq \lambda_{1}$.  And like before
$\frac{-\sum_{i\neq j}h_{ii1}h_{jj1}}{\lambda_{1}}\geq -C(1+\lambda_{1})$, $C\sim n$, $k$, $\inf(f_{0})$, $||f_{0}||_{C^{1}}$, so now $(\Re)$ becomes
\begin{eqnarray*}
0\geq -C\lambda_{1}+(\beta\phi^{'}-1)\lambda_{1}-C
\end{eqnarray*}
By choosing $\beta$ large enough, depending on $n$, $k$, $\inf(f_{0})$, $||f_{0}||_{C^{1}}$, we have our estimate.
\par
$\mathbf{k=1}$: At this case, $\sigma_{1}^{ii}=1$, therefore, from $(\Re)$ or $(\clubsuit 1)$, we have
\begin{eqnarray*}
0\geq \frac{1}{4}\frac{a}{u-a}h_{11}^{2}-C
\end{eqnarray*}
we also have the estimate.

\bigskip

$\mathbf{C^{2}}$  $\mathbf{bound}$ $\mathbf{for}$ $\mathbf{\rho_{ij}}$.
\par
Since we already proved curvature bound $\lambda_{1}\leq C$,
therefore, we also have
\begin{eqnarray*}
|\tilde{h_{j}^{i}}|=|\frac{1}{\phi^{2}\sqrt{\phi^{2}+|\nabla \rho|^{2}}}(\delta^{ik}-\frac{\rho_{i}\rho_{k}}{\tilde{\omega}(\tilde{\omega}+\phi)})(-\phi\rho_{kl}+2\phi^{'}\rho_{k}\rho_{l}+\phi^{2}\phi^{'}\delta_{kl})(\delta^{lj}-
\frac{\rho_{l}\rho_{j}}{\tilde{\omega}(\tilde{\omega}+\phi)})|
\end{eqnarray*}
is bounded, where $\tilde{\omega}=\sqrt{\phi^{2}+|\nabla \rho|^{2}}$.
\par
Let $G=\{\delta^{ik}-\frac{\rho_{i}\rho_{k}}{\tilde{\omega}(\tilde{\omega}+\phi)}\}$, then
$G^{-1}=\{\delta^{ik}+\frac{\rho_{i}\rho_{k}}{\phi(\tilde{\omega}+\phi)}\}$.
Since
\begin{eqnarray*}
\{-\phi\rho_{kl}+2\phi^{'}\rho_{k}\rho_{l}+\phi^{2}\phi^{'}\delta_{kl}\}=\phi^{2}\tilde{\omega}
\cdot G^{-1}\cdot \{\tilde{h_{j}^{i}}\}\cdot G^{-1}
\end{eqnarray*}
This implies
\begin{eqnarray*}
|-\phi\rho_{kl}+2\phi^{'}\rho_{k}\rho_{l}+\phi^{2}\phi^{'}\delta_{kl}|\leq C
\end{eqnarray*}
and hence
\begin{eqnarray*}
|\rho_{kl}|\leq C\qquad \forall k,l
\end{eqnarray*}

\section{Existence and uniqueness of prescribed curvature measure problem}

We now use continuous method to prove main theorem.
\par
\begin{proof}: For any positive function $f\in C^{2}(S^{n})$. For $0\leq t\leq 1$, set
\begin{eqnarray*}
f_{t}(x)=((1-t)+tf(x))^{\frac{1}{k}}
\end{eqnarray*}
Consider the following family of equations for $0\leq t\leq 1$
\begin{equation} \label{6.1}
(\phi^{n-1}\cdot\sqrt{\phi^{2}+|\nabla\rho|^{2}})^{\frac{1}{k}}\cdot \sigma_{k}^{\frac{1}{k}}(\kappa_{1},\cdots,\kappa_{n})
=f_{t}(x)
\end{equation}
Define: $I=\{t\in [0,1]:$ such that (\ref{6.1}) is solvable $\}$
\par
$\mathbf{(1)}$ Firstly, we could always find some constant $\rho_{0}$, such that
\begin{eqnarray*}
C_{n}^{k}\cdot \phi(\rho_{0})^{n-k}(\phi^{'})(\rho_{0})^{k}=1\qquad k<n
\end{eqnarray*}
This implies $\rho\equiv \rho_{0}$ is a solution when t=0. Therefore, I is nonempty.

\par
$\mathbf{(2)}$ $\quad$ Openness
\par
Like before, (\ref{6.1}) is equivalent as
\begin{equation} \label{6.2}
\frac{\phi^{n-k}}{\omega^{k-1}}\sigma_{k}(b_{j}^{i})=f
\end{equation}
The linearized operator at $\gamma$ is
\begin{equation} \label{6.3}
L_{1}(v)=\frac{\phi^{n-k}}{\omega^{k-1}}\cdot\widetilde{F^{st}}\cdot v_{st}+\sum_{l}b_{l}v_{l}-cv
\end{equation}
Here $\widetilde{F^{st}}$ was defined in the proof of $C^{1}$ estimate, $b_{l}$ is a function of $\gamma,\nabla\gamma,
\gamma_{ij}$, independent of v and
\begin{equation} \label{6.4}
c=(n-k)\frac{\phi^{n-k}\phi^{'}}{\omega^{k-1}}\cdot\sigma_{k}+\frac{\phi^{n-k+2}}{\omega^{k-1}}
\sum_{s,t}(\widetilde{F^{st}}\gamma_{s}\gamma_{t}+\widetilde{F^{st}}\delta_{st})
\end{equation}
Apparently, c is a positive number.
\par
Actually, we can also directly work with $\rho$:
\begin{eqnarray*}
\phi^{n-1}\sqrt{\phi^{2}+|\nabla \rho|^{2}}\sigma_{k}(\tilde{h_{j}^{i}})=f
\end{eqnarray*}
Then the linearized operator at $\rho$ is
\begin{equation} \label{6.5}
L_{2}(v)=\frac{\phi^{n-k}}{\tilde{\omega}^{k-1}}\cdot\hat{F}^{st}\cdot (\frac{v}{\phi})_{st}+
\sum_{l}\hat{b}_{l}(\frac{v}{\phi})_{l}-\hat{c}\frac{v}{\phi}
\end{equation}
Here $\tilde{\omega}=\sqrt{1+\frac{|\nabla\rho|^{2}}{\phi^{2}}}$,
\begin{eqnarray*}
\hat{F}^{st}=\sum_{i,j}\frac{\partial\sigma_{k}(\hat{h_{j}^{i}})}{\partial\hat{h_{j}^{i}}}
(\delta^{is}-\frac{\frac{\rho_{i}}{\phi}\frac{\rho_{s}}{\phi}}{\tilde{\omega}(\tilde{\omega}+1)})
(\delta^{tj}-\frac{\frac{\rho_{t}}{\phi}\frac{\rho_{j}}{\phi}}{\tilde{\omega}(\tilde{\omega}+1)})
\end{eqnarray*}
and $\hat{h_{j}^{i}}=\sqrt{\phi^{2}+|\nabla\rho|^{2}}\cdot\tilde{h_{j}^{i}}$. $\{\hat{F}^{st}\}$ is positive definite. $\hat{b}_{l}$ is a function of $\rho,\nabla\rho,
\rho_{ij}$, independent of v and
\begin{equation} \label{6.6}
\hat{c}=(n-k)\frac{\phi^{n-k}\phi^{'}}{\tilde{\omega}^{k-1}}\cdot\sigma_{k}+\frac{\phi^{n-k}}{\tilde{\omega}^{k-1}}
\sum_{s,t}(\hat{F}^{st}\rho_{s}\rho_{t}+\phi^{2}\hat{F}^{st}\delta_{st})
\end{equation}
Also, $\hat{c}$ is a positive number.

\par
Both linearized operators satisfy if $ L_{i}(v_{1})=L_{i}(v_{2}),\quad i=1,2$, then we have $v_{1}=v_{2}$. Then using inverse function theorem, we prove the openness, i.e. I is open.
\par
$\mathbf{(3)}$ $\quad$ Closeness
\par
Using Lemma 3 and Lemma 4 and the fact we already obtain $C^{2}$ estimate for $\rho$, we have equation (\ref{6.1})
is uniformly elliptic and concave. We apply the Evans-Krylov theorem and Schauder theorem to obtain
\begin{eqnarray*}
||\rho||_{C^{3,\alpha}(S^{n})}\leq C
\end{eqnarray*}
$C\sim n,k, \inf(f), ||f||_{C^{2}}$ and $\alpha$.
\par
This proves the closeness. i.e. I is close.
Combining ($\mathbf{1}$),($\mathbf{2}$),($\mathbf{3}$), we have $1\in I$. This finishes the proof of existence of prescribed curvature measure problem.
\end{proof}

\medskip

Proof of uniqueness:
\par
\begin{proof} Suppose $\gamma_{1}$,$\gamma_{2}$ both satisfy equation
\begin{eqnarray*}
\frac{\phi^{n-k}}{\omega^{k-1}}\sigma_{k}(b_{j}^{i})=f
\end{eqnarray*}

where $\omega=\sqrt{1+|\nabla \gamma|^{2}}$ and
\begin{eqnarray*}
b_{j}^{i}=(\delta^{ik}-\frac{\gamma_{i}\gamma_{k}}{\omega(\omega+1)})
(-\gamma_{kl}+\phi^{'}\gamma_{k}\gamma_{l}+\phi^{'}\delta_{kl})(\delta^{lj}-\frac{\gamma_{l}\gamma_{j}}{\omega(\omega+1)})
\end{eqnarray*}
\par
Let $\gamma(t)=(1-t)\gamma_{1}+t\gamma_{2}$, $\rho(t)$ be the relevant function with respect to $\gamma(t)$
satisfying $\frac{d\gamma}{d\rho}=\frac{1}{\phi}$. Denote $\rho_{2}=\rho(1),\rho_{1}=\rho(0)$ relate to
$\gamma_{2},\gamma_{1}$ respectly. Define
\begin{eqnarray*}
\breve{b_{j}^{i}}=(\delta^{ik}-\frac{(\gamma_{1})_{i}(\gamma_{1})_{k}}{\omega_{1}(\omega_{1}+1)})
(-(\gamma_{1})_{kl}+\phi^{'}(\rho_{1})(\gamma_{1})_{k}(\gamma_{1})_{l}+\phi^{'}(\rho_{1})\delta_{kl})
(\delta^{lj}-\frac{(\gamma_{1})_{l}(\gamma_{1})_{j}}{\omega_{1}(\omega_{1}+1)})
\end{eqnarray*}
\begin{eqnarray*}
\tilde{b_{j}^{i}}=(\delta^{ik}-\frac{(\gamma_{2})_{i}(\gamma_{2})_{k}}{\omega_{2}(\omega_{2}+1)})
(-(\gamma_{2})_{kl}+\phi^{'}(\rho_{2})(\gamma_{2})_{k}(\gamma_{2})_{l}+\phi^{'}(\rho_{2})\delta_{kl})
(\delta^{lj}-\frac{(\gamma_{2})_{l}(\gamma_{2})_{j}}{\omega_{2}(\omega_{2}+1)})
\end{eqnarray*}
here $\omega_{i}=\sqrt{1+|\nabla \gamma_{i}|^{2}}\qquad i=1,2$. Let
\begin{eqnarray*}
\hat{b_{j}^{i}(t)}=t\tilde{b_{j}^{i}}+(1-t)\breve{b_{j}^{i}}
\end{eqnarray*}
then
\begin{eqnarray*}
\{\hat{b_{j}^{i}(t)}\}=t\{\tilde{b_{j}^{i}}\}+(1-t)\{\breve{b_{j}^{i}}\}
\end{eqnarray*}
this means $\{\hat{b_{j}^{i}(t)}\}\in\Gamma_{k}\qquad \forall t\in [0,1]$. Define function
\begin{eqnarray*}
G(t)=\frac{\phi^{n-k}(\rho(t))}{\omega_{t+1}^{k-1}}\sigma_{k}(\hat{b_{j}^{i}(t)})
\end{eqnarray*}
Where $\omega_{t+1}=\sqrt{1+|\nabla \gamma(t)|^{2}}$, then we have
\begin{equation} \label{6.23}
0=G(1)-G(0)=\int_{0}^{1}G^{'}(s)ds
\end{equation}
A simple calculation shows
\begin{equation} \label{6.24}
G^{'}(s)=(n-k)\frac{\phi^{n-k}\phi^{'}}{\omega_{s+1}^{k-1}}\sigma_{k}\cdot(\gamma_{2}-\gamma_{1})+
\frac{\phi^{n-k}}{\omega_{s+1}^{k-1}}\sum_{i,j}\sigma_{k}^{ij}(\tilde{b_{j}^{i}}-\breve{b_{j}^{i}})
+\sum b_{l}(\gamma_{2}-\gamma_{1})_{l}
\end{equation}
here $\sigma_{k}^{\alpha\beta}=\frac{\partial\sigma_{k}}{\partial\hat{b_{\beta}^{\alpha}(s)}}$.
\par
Since $(\delta^{ik}-\frac{(\gamma_{1})_{i}(\gamma_{1})_{k}}{\omega_{1}(\omega_{1}+1)})
-(\delta^{ik}-\frac{(\gamma_{2})_{i}(\gamma_{2})_{k}}{\omega_{2}(\omega_{2}+1)})=\sum_{l}c_{l}(\gamma_{2}-\gamma_{1})_{l}$,
let
\begin{eqnarray*}
\check{F^{kl}}=\sum_{i,j}\sigma_{k}^{ij}(\delta^{ik}-\frac{(\gamma_{1})_{i}(\gamma_{1})_{k}}{\omega_{1}(\omega_{1}+1)})
(\delta^{lj}-\frac{(\gamma_{1})_{l}(\gamma_{1})_{j}}{\omega_{1}(\omega_{1}+1)})
\end{eqnarray*}
Then we obtain
\begin{eqnarray*}
& &\sum_{ij}\sigma_{k}^{ij}(\tilde{b_{j}^{i}}-\breve{b_{j}^{i}})=[\sum_{ij}\sigma_{k}^{ij}
(\delta^{ik}-\frac{(\gamma_{1})_{i}(\gamma_{1})_{k}}{\omega_{1}(\omega_{1}+1)})
(\delta^{lj}-\frac{(\gamma_{1})_{l}(\gamma_{1})_{j}}{\omega_{1}(\omega_{1}+1)})\\
& &\cdot(-(\gamma_{2})_{kl}+\phi^{'}(\rho_{2})(\gamma_{2})_{k}(\gamma_{2})_{l}+\phi^{'}(\rho_{2})\delta_{kl}
+(\gamma_{1})_{kl}-\phi^{'}(\rho_{1})(\gamma_{1})_{k}(\gamma_{1})_{l}-\phi^{'}(\rho_{1})\delta_{kl})]+\sum_{l}d_{l}(\gamma_{2}-\gamma_{1})_{l}\\
& &=\check{F^{kl}}(-(\gamma_{2}-\gamma_{1})_{kl}+(\phi^{'}(\rho_{2})-\phi^{'}(\rho_{1}))((\gamma_{1})_{k}(\gamma_{1})_{l}+\delta_{kl}))
+\sum_{l}e_{l}(\gamma_{2}-\gamma_{1})_{l}
\end{eqnarray*}
Therefore, (\ref{6.24}) becomes
\begin{eqnarray} \label{6.25}
G^{'}(s)&=&-\frac{\phi^{n-k}}{\omega_{s+1}^{k-1}}\check{F^{kl}}(\gamma_{2}-\gamma_{1})_{kl}
+\frac{\phi^{n-k}}{\omega_{s+1}^{k-1}}\check{F^{kl}}((\gamma_{1})_{k}(\gamma_{1})_{l}+\delta_{kl})\cdot\phi^{''}(\tilde{\rho})(\rho_{2}-\rho_{1})\notag\\
& &+(n-k)\frac{\phi^{n-k}\phi^{'}}{\omega_{s+1}^{k-1}}\sigma_{k}(\gamma_{2}-\gamma_{1})+\sum_{l}g_{l}(\gamma_{2}-\gamma_{1})_{l}
\end{eqnarray}
Where $\tilde{\rho}$ is a function whose value is between $\rho_{1}$ and $\rho_{2}$. i.e. if $\rho_{1}\leq\rho_{2}$, we have
$\rho_{1}\leq\tilde{\rho}\leq\rho_{2}$, otherwise $\rho_{2}\leq\tilde{\rho}\leq\rho_{1}$. This implies $\phi^{''}(\tilde{\rho})>0$.
Besides, $b_{l},c_{l},d_{l},e_{l},g_{l}$ are certain functions.
\par
Now, using (\ref{6.23}), we obtain
\begin{eqnarray} \label{6.26}
0&=&-(\int_{0}^{1}\frac{\phi^{n-k}}{\omega_{s+1}^{k-1}}\check{F^{kl}}ds)(\gamma_{2}-\gamma_{1})_{kl}
+(\int_{0}^{1}\frac{\phi^{n-k}}{\omega_{s+1}^{k-1}}\check{F^{kl}}ds)((\gamma_{1})_{k}(\gamma_{1})_{l}
+\delta_{kl})\cdot\phi^{''}(\tilde{\rho})(\rho_{2}-\rho_{1})\notag\\
& &+(\int_{0}^{1}(n-k)\frac{\phi^{n-k}\phi^{'}}{\omega_{s+1}^{k-1}}\sigma_{k}ds)
(\gamma_{2}-\gamma_{1})+\sum_{l}(\int_{0}^{1}g_{l}ds)(\gamma_{2}-\gamma_{1})_{l}
\end{eqnarray}
Now, if $\gamma_{2}-\gamma_{1}$ could reach its positive maximum point, say $x_{0}$, then at $x_{0}$, $\gamma_{2}>\gamma_{1}$,$\rho_{2}>\rho_{1}$,
and $\{\check{F^{kl}}\}$ is positive definite. Then the right side of (\ref{6.26}) is bigger than 0, which is impossible.
\par
Similarly, $\gamma_{2}-\gamma_{1}$ couldn't reach its negative minimal point. This implies $\gamma_{2}-\gamma_{1}\equiv 0$, i.e
$\gamma_{2}\equiv\gamma_{1}$. We prove the uniqueness.
\end{proof}

\bigskip

$\mathbf{Existence}$ $\mathbf{of}$ $\mathbf{prescribed}$ $\mathbf{0-th}$ $\mathbf{curvature}$ $\mathbf{measure}$.

\medskip

At this case, we obtain the existence and uniqueness of convex body with prescribed 0-th curvature measures  under extra condition $\inf_{S^{n}}(f)>1$. Here $f$ is the given function defined on $S^{n}$. To be specific, we prove the following theorem:

\begin{theorem} \label{Gauss Curvature}
 Suppose $f\in C^{2}(S^{n})$ and $\inf_{S^{n}}(f)>1$. Then
there exists a unique convex hypersurface $M\in C^{3,\alpha}$. Such that it satisfies
\begin{equation} \label{Gauss equation}
\sigma_{n}(\kappa_{1},\cdots,\kappa_{n})=\frac{f}{\phi(\rho)^{n-1}\sqrt{\phi(\rho)^{2}+|\nabla \rho|^{2}}}
\end{equation}
\par
Moreover, there is a constant C only depending on n,$||f||_{C^{2}}$, $\inf(f)$, and $\alpha$, such that
\begin{eqnarray*}
||\rho||_{C^{3,\alpha}}\leq C
\end{eqnarray*}
\end{theorem}

\begin{proof}
Repeating the proof of Theorem(\ref{T2}), we obtain $C^{0}$ estimate with extra condition $\inf_{S^{n}}(f)>1$:
 \begin{eqnarray*}
c_{0}\leq \min(\rho)\leq \max(\rho)\leq c_{1}
\end{eqnarray*}
$c_{0}$, $c_{1}\sim \inf(f)$, $|f|_{C^{0}}$, $n$.
\par
Now we proof $C^{1}$ estimate. We first introduce a new variable $\tilde{\gamma}$ satisfying
\begin{eqnarray*}
\frac{d\tilde{\gamma}}{d\rho}=\frac{1}{\phi^{3}}
\end{eqnarray*}
Then we have
\begin{eqnarray*}
\rho_{i} &= &\phi^{3}\tilde{\gamma}_{i}\\
\rho_{ij} &= & \phi^{3}\tilde{\gamma}_{ij}+3\phi^{'}\phi^{5}\tilde{\gamma}_{i}\tilde{\gamma}_{j}
\end{eqnarray*}
\par
Assume $|\nabla\tilde{\gamma}|^{2}$ obtains its maximum at $x_{0}$. Choose normal frame $\{e_{1},\cdots,e_{n}\}$ at $x_{0}$, satisfying
$\tilde{\gamma}_{1}=|\nabla\tilde{\gamma}|, \tilde{\gamma}_{i}=0$, $\forall i\geq 2$. Rotate $\{e_{2},\cdots,e_{n}\}$, such that
$\{\tilde{\gamma}_{ij}\}$ is diagonal at $x_{0}$, $n\geq i,j\geq 2$. At $x_{0}$,
\begin{eqnarray*}
(|\nabla \tilde{\gamma}|^{2})_{\alpha}=\sum_{k}2\tilde{\gamma}_{k}\tilde{\gamma}_{k\alpha}
=2\tilde{\gamma}_{1}\tilde{\gamma}_{1\alpha}=0
\end{eqnarray*}
This implies $\tilde{\gamma}_{1\alpha}=0$, $\forall \alpha$. Hence, at $x_{0}$,
 \begin{eqnarray*}
h_{j}^{i} &= &\frac{1}{\phi^{2}\sqrt{\phi^{2}+|\nabla \rho|^{2}}}(\delta^{ik}-\frac{\rho_{i}\rho_{k}}{\phi^{2}+|\nabla \rho|^{2}})(-\phi\rho_{jk}+2\phi^{'}\rho_{k}\rho_{j}+\phi^{2}\phi^{'}\delta_{kj})\\
          &= &\frac{1}{\phi\sqrt{1+\phi^{4}|\nabla \tilde{\gamma}|^{2}}}(\delta^{ik}-\phi^{4}\frac{\tilde{\gamma}_{i}\tilde{\gamma}_{k}}{1+\phi^{4}|\nabla\tilde{\gamma}|^{2}})
          (-\phi^{2}\tilde{\gamma}_{jk}-\phi^{'}\phi^{4}\tilde{\gamma}_{k}\tilde{\gamma}_{j}+\phi^{'}\delta_{kj})\\
          &= & \frac{1}{\phi\sqrt{1+\phi^{4}|\nabla \tilde{\gamma}|^{2}}}(-\phi^{2}\tilde{\gamma}_{ij}+\phi^{'}\delta_{ij}
          -\frac{2}{1+\phi^{4}|\nabla\tilde{\gamma}|^{2}}\phi^{'}\phi^{4}\tilde{\gamma}_{i}\tilde{\gamma}_{j})
\end{eqnarray*}
Therefore
\begin{eqnarray*}
h_{1}^{1} &= &\frac{1}{\phi\sqrt{1+\phi^{4}|\nabla \tilde{\gamma}|^{2}}}\cdot\frac{\phi^{'}-\phi^{'}\phi^{4}|\nabla\tilde{\gamma}|^{2}}
{1+\phi^{4}|\nabla \tilde{\gamma}|^{2}}\\
h_{i}^{i} &= &\frac{1}{\phi\sqrt{1+\phi^{4}|\nabla \tilde{\gamma}|^{2}}}(-\phi^{2}\tilde{\gamma}_{ii}+\phi^{'})\quad \forall i\geq 2
\end{eqnarray*}
Besides, $h_{1}^{i}=h_{i}^{1}=0$, $\forall i\geq 2$ and $h_{i}^{j}=0$, $\forall i\neq j\geq 2$.
\par
Since the eigenvalues of $\{h_{j}^{i}\}$ are positve, we have $h_{1}^{1}>0$. We then obtain
\begin{eqnarray*}
\frac{1}{\phi^{2}}>|\nabla\tilde{\gamma}|
\end{eqnarray*}
We proof $C^{1}$ estimate.
\par
$C^{2}$ estimate follows from Theorem \ref{C2}. As for the existence, define
\begin{eqnarray*}
f_{t}(x)=(2(1-t)+tf(x))^{\frac{1}{n}}
\end{eqnarray*}
$f_{t}(x)$ satisfies $\inf_{S^{n}}f_{t}>1$, when $0\leq t\leq 1$.
\par
Consider the following family of equations for $0\leq t\leq 1$
\begin{equation} \label{Gaussex}
(\phi^{n-1}\cdot\sqrt{\phi^{2}+|\nabla\rho|^{2}})^{\frac{1}{n}}\cdot \sigma_{n}^{\frac{1}{n}}(\kappa_{1},\cdots,\kappa_{n})
=f_{t}(x)
\end{equation}
\par
Define: $\tilde{I}=\{t\in [0,1]:$ such that (\ref{Gaussex}) is solvable $\}$
\par
We could always find some constant $\rho_{0}$, such that
\begin{eqnarray*}
 (\phi^{'})(\rho_{0})^{n}=2
\end{eqnarray*}
This implies $\rho\equiv \rho_{0}$ is a solution when t=0. Therefore, $\tilde{I}$ is nonempty. Openness follows from
(\ref{6.3}),(\ref{6.4}),(\ref{6.5}),(\ref{6.6}). Using $C^{2}$ estimate, Lemma \ref{lemma3} and Lemma \ref{lemma4},
we can see equation (\ref{Gaussex})
is uniformly elliptic and concave. We apply the Evans-Krylov theorem and Schauder theorem to obtain the regularity and closeness.
Therefore we prove the existence. The uniqueness follows from (\ref{6.26}).

\end{proof}


\begin{thebibliography}{}


\bibitem{ex1} A.D.ALEXANDROV, {\em Zur theorie der gemishchten volumina von knovexen korpern, 2},
Math.USSR-Sb.2(1937),1205-1238.(1927)

\bibitem{1} A.D.ALEXANDROV, {\em Existence and uniqueness of a convex surface with a given integral curvature},
C.R.(Doklady) Acad.Sci.URSS(N.S)
35 (1942), 131-134. MR 0007625.

\bibitem{2} CARL B. ALLENDOERFER, {\em Steiner's formulae on a general $S^{n+1}$}. Bull.Amer.Math.Soc.,
54:128-135,1948.

\bibitem{ex3} C.BERG, {\em Corps convexes et potentiels sph$\acute{e}$riques}, Mat.-Fys.Medd.Danske Vid. Selsk. 37(1969), no.6,64 pp
. MR 0254789 (1927)

\bibitem{brendle} SIMON BRENDLE, KYEONGSU CHOI and PANAGIOTA DASKALOPOULOS, {\em Asymptotic behavior of flows by
powers of the Gaussian curvature}, Acta Math, 219 (2017), 1-16.

\bibitem{3} L.CAFFARELLI,L.NIRENBERG,and J.SPRUCK, {\em The Dirichlet problem for nonlinear second order elliptic
equations, III: Functions of the eigenvalues of the Hessian}, Acta Math. 155 (1985), 261-301.
MR 0806416

\bibitem{4} L.CAFFARELLI,L.NIRENBERG,and J.SPRUCK, {\em "Nonlinear second order elliptic equations IV:
Starshaped compact Weingarten hypersurfaces" in Current Topics in Partial Differential Equations},
Kinokuniya, Tokyo, 1986, 1-26. MR 1112140

\bibitem{ex6} S.Y.CHENG, S.T.YAU, {\em On the regularity of the solution of the n-dimensional Minkowski problem},
Comm.Pure Appl. Math.29(1976),495-516.MR 0423267 (1927)

\bibitem{5} H.FEDERER, {\em Curvature measures}, Trans.Amer.Math.Soc. 93 (1959), 418-491. MR 0110078

\bibitem{ex8} W.J.FIREY, {\em Christoffels's problem for general convex bodies}, Mathematika 15(1968), 7-21.
MR 0230259(1927)

\bibitem{6} P.GUAN, {\em Curvature measures, isoperimetric type inequalities and fully nonlinear PDEs},
Fully Nonlinear PDEs in Real and Complex Geometry and Optics, 47-94, 2014

\bibitem{Junfang} P.GUAN, J.LI, {\em A mean curvature flow in space form}, International Mathematics Research Notices,
Vol.2015, NO.13,(2015) 4716-4740

\bibitem{7} P.GUAN, JUNFANG LI and YANYAN LI, {\em Hypersurfaces of prescribed curvature measure}
Duke Math.J. Vol.161.No.10(2012), 1927-1942

\bibitem{8} P.GUAN, Y. LI, {\em $C^{1,1}$ estimates for solutions of a problem of Alexandrov}, Comm.
Pure Appl. Math. 50(1997), 789-811. MR 1454174

\bibitem{9} P.GUAN, C.LIN, and XINAN MA, {\em The existence of convex body with prescribed curvature measures},
Int.Math.Res.Not.IMRN 2009, no.11,1947-1975. MR 2507106

\bibitem{ex12} P.GUAN, XINAN.MA, {\em The Christoffel-Minkowski problem, I: Convexity of solutions of a Hessian equation},
Invent.Math.151(2003),553-577.MR 1961338 (1927,1928)

\bibitem{10} P.GUAN, C.REN and Z.WANG, {\em Global $C^{2}$ estimates for convex solutions of curvature equations},
Communications on Pure and Applied Mathematics, Vol.LXVIII,(2015) 1287-1325

\bibitem{11} QINIAN JIN, YANYAN LI, {\em Starshaped compact hypersurfaces with prescribed m-th
mean curvature in hyperbolic space}, Discrete and Continuous Dynamical Systems 15 (2006), 367-377.

\bibitem{12} PETER KOHLMANN, {\em Curvature measures and Steiner formulae in space forms}. Geom.Dedicata,
40(2):191-211,1991

\bibitem{ex14} H.LEWY, {\em On differential geometry in the large, I: Minkowski's problem},
Trans.Amer.Math.Soc.43 (1938),258-270.MR 1501942(1927)

\bibitem{ex15} L.NIRENBERG, {\em The Weyl and Minkowski problems in differential geometry in the large},
Comm.Pure Appl.Math.6(1953),337-394.MR 0058265 (1927)

\bibitem{13} V.I.OLIKER, {\em Existence and uniqueness of convex hypersurfaces with prescribed Gausssian
curvature in spaces of constant curvature}, Sem.Inst.Mate.Appl. "Giovanni Sansone", Univ. Studi Firenze,
1983, 1-39.

\bibitem{ex17} A.V.POGORELOV, {\em Regularity of a convex surface with given Gaussian curvature},
Mat.Sbornik N.S.31(73)(1952),88-103.MR 0052807(1927)

\bibitem{14} A.V.POGORELOV, {\em Extrinsic Geometry of Convex Surfaces}, Transl.Math.Monogr.35,
Amer.Math.Soc.,Providence, 1973. MR 0346714

\bibitem{ex19} A.V.POGORELOV, {\em The Minkowski Multidimensional Problem},
V.H. Winston, Washington,D.C., 1978.MR 0478079 (1927)

\bibitem{RW} CHANGYU REN, ZHIZHANG WANG, {\em On the curvature estimates for Hessian equations},
American Journal of Mathematics 141(5), 1281-1315, 2019

\bibitem{15} R.SCHNEIDER, {\em Convex bodies: The Brunn-Minkowski Theory}. Encyclopedia Math. Appl. 44
, Cambridge Univ.Press, Cambridge, 1993. MR 1216521

\bibitem{spruckX} J.SPRUCK, L.XIAO, {\em A note on star-shaped compact hypersurfaces with prescribed scalar curvature
in space forms}, Rev. Mat. Iberoam. 33(2017), 547-554.

\bibitem{16} GIONA VERONELLI, {\em Boundary structure of convex sets in the hyperbolic space}, Monatshefte f$\ddot{u}$r Mathematik, 188(3)
, 567-586, 2019






\end{thebibliography}
\end{document}